\documentclass[12pt,reqno]{amsart}
\usepackage[cp1251]{inputenc}
\usepackage[T2A]{fontenc}
\usepackage[english]{babel}
\usepackage{amsmath,amsfonts,amssymb}
\usepackage{geometry}
\usepackage{amsmath, amsthm, amscd, amsfonts, amssymb, graphicx, color}
\usepackage{cite}
\usepackage[bookmarksnumbered, colorlinks, plainpages]{hyperref}

\numberwithin{equation}{section}

\newtheorem{theorem}{Theorem}[section]
\newtheorem{corollary}[theorem]{Corollary}
\newtheorem{lemma}[theorem]{Lemma}
\newtheorem{proposition}[theorem]{Proposition}

\theoremstyle{remark}
\newtheorem{remark}[theorem]{Remark}
\theoremstyle{definition}

\title[Extended Sobolev scale for vector bundles]
{Extended Sobolev scale for vector bundles,\\ and its applications}

\author[A. Murach]{Aleksandr Murach}

\address{Institute of Mathematics of the National Academy of Sciences of Ukraine, 3 Tereshchen\-kivs'ka Street, Kyiv, 01024, Ukraine}

\email{murach@imath.kiev.ua}

\author[T. Zinchenko]{Tetiana Zinchenko}

\address{Grace-Hopper-Gesamtschule Teltow, 146 Mahlover Street, Teltow, 14513, Germany}

\email{Tetiana.zinger@icloud.com}



\begin{document}
\renewcommand{\rho}{\varrho}
\maketitle

\begin{abstract}
We study an extended Sobolev scale for smooth vector bundles over a smooth closed manifold. This scale is built on the base of inner product distribution spaces of generalized smoothness given by an arbitrary positive function OR-varying at infinity. We show that this scale is obtained by the quadratic interpolation (with a function parameter) between inner product Sobolev spaces, is closed with respect to the quadratic interpolation, and consists  of all Hilbert spaces that are interpolation spaces between inner product Sobolev spaces. Embedding theorems and a duality theorem are proved for this scale. We give applications of the extended Sobolev scale to mixed-order (Douglis--Nirenberg) elliptic pseudodifferential operators acting between vector bundles of the same rank. We prove their Fredholm property on appropriate pairs of spaces on the scale, give a sufficient and necessary condition for the local generalized smoothness of solutions to a mixed-order elliptic system and provide a corresponding \textit{a priori} estimate of the solutions. We also give a sufficient condition for a chosen component of the solution to be $q$ times continuously differentiable on a subset of the manifold.
\end{abstract}

\section{Introduction}

An extended Sobolev scale over $\mathbb{R}^{n}$ and over smooth closed manifolds was introduced and investigated by Mikhailets and Murach \cite{MikhailetsMurach09RepNASU3, MikhailetsMurach13UMJ3, MikhailetsMurach14, MikhailetsMurach15ResMath1}. It is built on the base of the H\"ormander spaces $B_{p,k}$ \cite[Section 2.2]{Hermander63} considered in the Hilbert isotropic case where $p=2$ and $k(\xi)\equiv \varphi(\langle\xi\rangle)$ for some function $\varphi: [1,\infty)\rightarrow(0,\infty)$ OR-varying  at infinity in the sense of Avakumovi\'c (as usual, $\langle\xi\rangle = (1+|\xi|)^{1/2}$ for $\xi\in \mathbb{R}^n$). This scale can be interpreted as the extended Hilbert scale generated by the operator $(1-\Delta)^{1/2}$ or by more general positive elliptic pseudodifferential operator of the first order \cite[Section~5]{MikhailetsMurachZinchenko25AdvMath}. Very recently it was introduced for the lattice $\mathbb{Z}^{n}$ by Milatovich \cite{Milatovic24JPsDOAppl, Milatovic24CAOT}. If $\varphi$ is a regularly varying function at infinity in the sense of Karamata, the spaces $B_{2,k}$ form the refined Sobolev scale \cite{MikhailetsMurach05UMJ5, MikhailetsMurach06UMJ3, MikhailetsMurach08MFAT1, MikhailetsMurach12BJMA2} as an important part of the above scale.

The above scales are formed by distribution spaces of generalized smoothness given by a function of frequency variables (or by a number sequence representing this function). Such a function parameter allows describing the smoothness of distributions in terms of their Fourier transform far more exactly than it is possible by means of classical distribution spaces whose smoothness is given by a single number. Distribution spaces of generalized smoothness have been actively studied in the last two decades (see, e.g., \cite{AlmeidaCaetano11, Baghdasaryan10, DominguezTikhonov23, FarkasLeopold06, HaroskeMoura08, HaroskeLeopoldMouraSkrzypczak23, LoosveldtNicolay19, MouraNevesSchneider14}). They have various applications in the approximation theory \cite{Stepanets05}, theory of stochastic process \cite{Jacob010205}, and theory of partial differential equations \cite{Faierman20, LosMikhailetsMurach21CPAA10, MikhailetsMurach14, NicolaRodino10}. Certainly, Hilbert spaces of generalized smoothness are of a special interest due to their applications in spectral theory of differential operators \cite{MikhailetsMurach24ProcEdin, MikhailetsMurachZinchenko25AdvMath}.

The extended Sobolev scale has the following important interpolation properties: it is obtained by the quadratic interpolation (with a function parameter) between inner product Sobolev spaces, is closed with respect to the quadratic interpolation, and consists  of all Hilbert spaces that are interpolation spaces between inner product Sobolev spaces. The first of these properties has played a key role in building a general theory of elliptic systems and elliptic boundary value problems on this  scale \cite{AnopChepurukhinaMurach21Axioms, AnopDenkMurach21CPAA, AnopKasirenko16MFAT4, AnopMurach14UMJ, AnopMurach14MFAT2, Murach09UMJ3, MurachZinchenko13MFAT1, ZinchenkoMurach12UMJ11, ZinchenkoMurach14JMathSci} (see also survey \cite{MikhailetsMurachChepurukhina25UMJ2}). Applications of the refined Sobolev scale to elliptic operators and elliptic problems are summarized in \cite{MikhailetsMurach12BJMA2, MikhailetsMurach14}.

In the modern theory of partial  differential equations and their applications, elliptic equations given on vector bundles over manifolds play an important role (see, e.g., \cite{Wells, Hermander07v3}). Therefore, it is useful to introduce and investigated a version of the extended Sobolev scale for these bundles. The present paper is devoted to such a task. We consider smooth vector bundles over a closed (compact) smooth manifold and show that the extended Sobolev scale is well defined for these bundles by the trivialization of the bundle and localization of the manifold and prove that this scale possesses the above-mentioned interpolation properties. We also prove embedding theorems and a duality theorem for the scale. These results are applied to the study of mixed-order (or Douglis--Nirenberg) elliptic pseudodifferential operators on a pair of vector bundles over the manifold. We show that such operators are Fredholm on appropriate pairs of spaces on the extended Sobolev scale. We give a sufficient and necessary condition for solutions of mixed-order elliptic systems to have a prescribed local smoothness on the scale and also establish a corresponding \textit{a priori} estimate of the solutions. We give a sufficient condition for a chosen component of the solution to be $q$ times continuously differentiable on a given subset of the manifold. This condition is exact on the scale. Note that the refined Sobolev scale for vector bundles was introduced and investigated by Zinchenko \cite{Zinchenko17OpenMath}.

\section{The extended Sobolev scale}\label{sec2}

In this section, we will introduce the extended Sobolev scale for a vector bundle over a smooth closed manifold.
We build this scale on the base of its analog for Euclidean space with the help of local charts on the manifold and local trivializations of the bundle. Therefore we recall the definition of the extended Sobolev scale for $\mathbb{R}^{n}$, which was introduced and investigated by Mikhailets and Murach \cite{MikhailetsMurach13UMJ3, MikhailetsMurach14, MikhailetsMurach15ResMath1}.

Let $1\leq n \in \mathbb{Z}$. The extended Sobolev scale over $\mathbb{R}^{n}$ consists of the inner product H\"ormander spaces
$H^{\varphi}(\mathbb{R}^{n})$ whose smoothness index $\varphi$ is given by an arbitrary function from the class $\mathrm{OR}$. By definition, this  class consists of all Borel measurable
functions $\varphi:[1,\infty)\rightarrow(0,\infty)$ for each of which there exist numbers $a>1$ and $c\geq1$ such that
\begin{equation}\label{f3}
c^{-1}\leq\frac{\varphi(\lambda t)}{\varphi(t)}\leq c\quad\mbox{for all}\quad
t\geq1\quad\mbox{and}\quad\lambda\in[1,a]
\end{equation}
(the numbers $a$ and $c$ may depend on $\varphi\in\mathrm{OR}$). These
functions are said to be $\mathrm{OR}$-varying at infinity. They were introduced by Avakumovi\'c \cite{Avakumovic36} in 1936, are well investigated, and have various applications (see, e.g., \cite{BinghamGoldieTeugels89, BuldyginIndlekoferKlesovSteinebach18, Seneta76}).

The class $\mathrm{OR}$ admits the following description \cite[Theorem A~1]{Seneta76}:
\begin{equation*}
\varphi\in\mathrm{OR}\quad\Longleftrightarrow\quad\varphi(t)=
\exp\Biggl(\beta(t)+\int\limits_{1}^{t}\frac{\gamma(\tau)}{\tau}\;d\tau\Biggr), \;\;t\geq1,
\end{equation*}
where the real-valued functions $\beta$ and $\gamma$ are Borel measurable and bounded on $[1,\infty)$. Note \cite[Theorem A~2]{Seneta76} that condition \eqref{f3} is equivalent to the two-sided inequality
\begin{equation}\label{1.1}
c^{-1}\lambda^{s_{0}}\leq\frac{\varphi(\lambda t)}{\varphi(t)}\leq
c\lambda^{s_{1}}\quad\mbox{for each}\quad t\geq1\quad\mbox{and}\quad\lambda\geq1,
\end{equation}
in which (another) constant $c\geq1$ is independent of $t$ and $\lambda$.
For every $\varphi\in\mathrm{RO}$, there exist lower and upper Matuszewska indexes \cite[Section 2.1.2]{BinghamGoldieTeugels89}:
\begin{equation}\label{1.2}
\sigma_{0}(\varphi):=
\sup\,\{s_{0}\in\mathbb{R}:\,\mbox{the left-hand side of \eqref{1.1} is true}\},
\end{equation}
\begin{equation}\label{1.3}
\sigma_{1}(\varphi):=\inf\,\{s_{1}\in\mathbb{R}:\,\mbox{the right-hand side of \eqref{1.1} is true}\}.
\end{equation}
Certainly, $-\infty<\sigma_{0}(\varphi)\leq\sigma_{1}(\varphi)<\infty$.

Let $\varphi\in\mathrm{OR}$. By definition, the complex linear space $H^{\varphi}(\mathbb{R}^{n})$ consists of all
distributions $w\in\mathcal{S}'(\mathbb{R}^{n})$ such that their Fourier transform
$\widehat{w}:=\mathcal{F}w$ is locally Lebesgue integrable over $\mathbb{R}^{n}$ and
satisfies the condition
\begin{equation*}
\int\limits_{\mathbb{R}^{n}}\varphi^2(\langle\xi\rangle)\,
|\widehat{w}(\xi)|^2\,d\xi
<\infty.
\end{equation*}
As usual, $\mathcal{S}'(\mathbb{R}^{n})$ denotes the linear topological space of all tempered distributions given in $\mathbb{R}^{n}$, and
$\langle\xi\rangle:=(1+|\xi|^{2})^{1/2}$ is the smoothed absolute value  of the frequency variable $\xi\in\mathbb{R}^{n}$. The inner product in $H^{\varphi}(\mathbb{R}^{n})$ is defined by the formula
\begin{equation*}
(w_1,w_2)_{H^{\varphi}(\mathbb{R}^{n})}:=
\int\limits_{\mathbb{R}^{n}}\varphi^2(\langle\xi\rangle)\,
\widehat{w_1}(\xi)\,\overline{\widehat{w_2}(\xi)}\,d\xi.
\end{equation*}
It endows $H^{\varphi}(\mathbb{R}^{n})$ with the Hilbert space structure and induces
the norm
\begin{equation*}
\|w\|_{H^{\varphi}(\mathbb{R}^{n})}:=
(w,w)_{H^{\varphi}(\mathbb{R}^{n})}^{1/2}.
\end{equation*}

The space $H^{\varphi}(\mathbb{R}^{n})$ is a special case of the space $\mathcal{B}_{p,k}$ introduced and investigated by H\"ormander \cite[Section 2.2]{Hermander63} (see also his monograph \cite[Section 10.1]{Hermander05v2}). Namely,  $H^{\varphi}(\mathbb{R}^{n}) = \mathcal{B}_{p,k}$ if $p=2$ and $k(\xi)\equiv \varphi(\langle\xi\rangle)$. Note that in the Hilbert case of $p=2$ the space $\mathcal{B}_{p,k}$ was also investigated by Volevich and Paneah \cite[\S~2]{VolevichPaneah65}.

In the case where $\varphi(t) \equiv t^s$ for some $s\in\mathbb{R}$, the space $H^{\varphi}(\mathbb{R}^{n})$ becomes the inner product Sobolev space $H^{(s)}(\mathbb{R}^{n})$ of order $s$. Generally, we have the dense continuous embeddings
\begin{equation}\label{2.34}
\begin{gathered}
H^{(s_1)}(\mathbb{R}^{n})\hookrightarrow
H^{\varphi}(\mathbb{R}^{n})\hookrightarrow
H^{(s_0)}(\mathbb{R}^{n})\\
\mbox{for arbitrary real}\;\; s_0<\sigma_{0}(\varphi)\;\; \mbox{and}\;\; s_1>\sigma_{1}(\varphi).
\end{gathered}
\end{equation}
The function parameter $\varphi$ is naturally said to be the smoothness index of the space $H^{\varphi}(\mathbb{R}^{n})$ (and its versions for manifolds and vector bundles). Following \cite{MikhailetsMurach13UMJ3, MikhailetsMurach14, MikhailetsMurach15ResMath1} (e.g., \cite[p.~105]{MikhailetsMurach14}), we call the class of distribution spaces
\begin{equation*}
\{H^{\varphi}(\mathbb{R}^{n}): \varphi\in\mathrm{OR}\}
\end{equation*}
the extended Sobolev scale over $\mathbb{R}^{n}$.

Let us now introduce its version for a vector bundle over a closed manifold. Let $\Gamma$ be a closed (i. e. compact and without boundary) infinitely smooth real manifold of dimension $n\geq1$. Suppose that a certain $C^\infty$-density $dx$ is given on $\Gamma$.
We choose a finite atlas belonging to the $C^{\infty}$-structure on $\Gamma$. Let this atlas consist of local charts $\alpha_{j}:\mathbb{R}^{n}\leftrightarrow\Gamma_{j}$,
$j=1,\ldots,\varkappa$, where the open sets  $\Gamma_{j}$ form a finite covering of $\Gamma$. We also choose functions $\chi_{j}\in C^{\infty}(\Gamma)$, $j=1,\ldots,\varkappa$,
that satisfy the condition $\mathrm{supp}\,\chi_{j}\subset\Gamma_{j}$ and form a partition of unity on $\Gamma$. Let $\pi: V \rightarrow\Gamma$ be an infinitely smooth complex vector bundle of rank $p\geq1$ over $\Gamma$. Here, $V$ is the total space of the bundle, $\Gamma$ is the base  space, and $\pi$ is the projector (see, e.g., \cite[Chapter~I, Section~2]{Wells}). We choose the set $\Gamma_{j}$ so that the local trivialization $\beta_{j}:\pi^{-1}(\Gamma_{j})\leftrightarrow
\Gamma_{j}\times\mathbb{C}^{p}$ is defined for each $j\in\{1,\ldots,\varkappa\}$.

Let $\mathcal{D}'(\Gamma, V)$ denote the linear topological space of all generalized sections of the vector bundle $\pi:\nobreak V \rightarrow\Gamma$. We interpret $\mathcal{D}'(\Gamma, V)$ as the antidual space to the linear topological space $C^\infty(\Gamma, V)$ of all infinitely smooth sections of this vector bundle. With every generalized section $u\in\mathcal{D}'(\Gamma, V)$, we associate a collection of vector-valued distributions $u^{\times}_j\in (\mathcal{D}'(\Gamma_j))^p$, where $j=1,~\ldots,~\varkappa$, in the following way: $u^{\times}_j(w):=u(w^{\circ}_j)$ for an arbitrary vector-valued function $w\in(C^\infty_0(\Gamma_j))^p$.
Here, the section $w^{\circ}_j\in C^\infty(\Gamma, V)$ is defined by the formula
\begin{equation*}
w^{\circ}_j(x) := \left\{
            \begin{array}{ll}
              \beta_j^{-1}(x,w(x)) & \hbox{if}\; x\in\Gamma_j; \\
              0 & \hbox{if} \;x\in\Gamma\setminus\Gamma_j.
            \end{array}
          \right.
\end{equation*}
As usual, $C^\infty_0(\Gamma_j):= \{v\in C^\infty(\Gamma_j): \mathrm{supp}\, v\subset\Gamma_j\}$, and $\mathcal{D}'(\Gamma_j)$ denotes the linear topological space of all distributions on $\Gamma_j$. We say that $u^{\times}_j$ is the representation of the generalized section $u$ in the local trivialization~$\beta_j$.

Let us introduce the space $H^{\varphi}(\Gamma,V)$ for arbitrary $\varphi\in\mathrm{OR}$. By definition, the linear space $H^{\varphi}(\Gamma, V)$ consists of all generalized sections
$u\in\nobreak\mathcal{D}'(\Gamma, V)$ such that $(\chi_{j}u^{\times}_j)\circ\alpha_{j}\in
(H^{\varphi}(\mathbb{R}^{n}))^p$ for every $j\in\{1,\ldots,\varkappa\}$. Here, $(\chi_{j}u^{\times}_j)\circ\alpha_{j}$ is the representation of the vector-valued distribution $\chi_{j}u^{\times}_{j}\in(\mathcal{D}'(\Gamma_j))^{p}$ in the local chart $\alpha_{j}$. We endow the space $H^{\varphi}(\Gamma,V)$ with the inner product
\begin{equation*}
(u, v)_{H^{\varphi}(\Gamma,V)}:=
\sum_{j=1}^\varkappa((\chi_{j}\,u^{\times}_{j})\circ\alpha_{j}, (\chi_{j}\,v^{\times}_{j})\circ\alpha_{j})_
{(H^{\varphi}(\mathbb{R}^{n}))^{p}},
\end{equation*}
where $u, v\in H^{\varphi}(\Gamma,V)$, and the corresponding norm
\begin{equation*}
\|u\|_{H^{\varphi}(\Gamma,V)}:=(u,u)^{1/2}_{H^{\varphi}(\Gamma,V)}=
\biggl(\,\sum_{j=1}^\varkappa
\|(\chi_{j}\,u^{\times}_{j})\circ\alpha_{j}\|^2_
{(H^{\varphi}(\mathbb{R}^{n}))^{p}}\biggr)^{1/2}.
\end{equation*}
The inner product and the corresponding norm in the Hilbert space $(H^{\varphi}(\mathbb{R}^{n}))^{p}$ are defined in the standard way.

The space $H^{\varphi}(\Gamma,V)$ is Hilbert and separable and does not depend up to equivalence of norms on our choice of the atlas $\{\alpha_j\}$, partition of unity $\{\chi_j\}$, and local trivializations $\{\beta_j\}$. This will be proved bellow as Theorems \ref{t4.2a} and~\ref{t4.2b}.

If $\varphi(t)\equiv t^s$ for some $s\in\mathbb{R}$, then $H^{\varphi}(\Gamma, V)$ becomes the inner product Sobolev space  $H^{(s)}(\Gamma, V)$ of order $s$ (see, e.g., \cite[Chapter~IV, Section~1]{Wells}).

The class of function spaces
\begin{equation}\label{1.36}
\{H^{\varphi}(\Gamma, V): \varphi\in\mathrm{OR}\}
\end{equation}
is called the extended Sobolev scale for the vector bundle $\pi:V\to\Gamma$ \cite[Section~2]{Zinchenko17Collection3}.

In the case of the trivial vector bundle of rank $p=1$, the space  $H^{\varphi}(\Gamma, V)$ consists of distributions on $\Gamma$ and is denoted by $H^{\varphi}(\Gamma)$. In this case the extended Sobolev scale was introduced and investigated in \cite{MikhailetsMurach09RepNASU3, MikhailetsMurach14, MikhailetsMurachZinchenko25AdvMath}.

\section{Quadratic interpolation with function parameter}\label{sec3}

According to \cite[Theorem~2.19]{MikhailetsMurach14}, the extended Sobolev scale over $\mathbb{R}^{n}$ is obtained by the quadratic interpolation with an appropriate function parameter between certain inner product Sobolev spaces. We will prove that this interpolation property is inherited by the extended Sobolev scale for the vector bundle $\pi:V\to\Gamma$ (see Theorem~\ref{t4.1}). This result and the above interpolation method will play a key role in our proofs of properties of this scale. Therefore we recall the definition of this method. It is sufficient for our purposes to restrict ourselves to the case of separable complex Hilbert spaces, we following monograph \cite[Section 1.1]{MikhailetsMurach14}. Note that this method appeared first in Foia\c{s} and Lions' paper \cite[Section 3.4]{FoiasLions61}.

Let $X:=[X_{0},X_{1}]$ be an ordered pair of separable complex Hilbert spaces $X_{0}$ and $X_{1}$ such that $X_{1}$ is a linear manifold in $X_{0}$ and that the embedding $X_{1}\hookrightarrow X_{0}$ is continuous and dense. This pair is called regular. For $X$ there exists a positive-definite self-adjoint operator $J$ in $X_{0}$ with the domain $X_{1}$ such that $\|Ju\|_{X_{0}}=\|u\|_{X_{1}}$ for arbitrary $u\in X_{1}$. The operator $J$ is uniquely determined by the pair $X$ and is called the generating operator for this pair.

Let $\mathcal{B}$ denote the set of all Borel measurable
functions $\psi:(0,\infty)\rightarrow(0,\infty)$ such that $\psi$ is bounded on every
compact interval $[a,b]$, with $0<a<b<\infty$, and that $1/\psi$ is bounded on every
set $[r,\infty)$, with $r>0$. Given $\psi\in\mathcal{B}$, consider the operator $\psi(J)$ defined as the Borel function $\psi$ of the self-adjoint operator $J$ with the help of Spectral Theorem (see, e.g., \cite[Chapter~XII, Section~2]{DunfordSchwartz88-II}).  The operator $\psi(J)$ is (generally) unbounded and positive-definite in $X_{0}$. Let $[X_{0},X_{1}]_{\psi}$ or, simply,
$X_{\psi}$ denote the domain of $\psi(J)$ endowed with the inner product
$(u_{1},u_{2})_{X_{\psi}}:=(\psi(J)u_{1},\psi(J)u_{2})_{X_{0}}$ and the
corresponding norm $\|u\|_{X_{\psi}}=\|\psi(J)u\|_{X_{0}}$. The space $X_{\psi}$ is
Hilbert and separable.

A function $\psi\in\mathcal{B}$ is called an interpolation parameter if the following property is satisfied for arbitrary regular pairs $X=[X_{0},X_{1}]$ and $Y=[Y_{0},Y_{1}]$ of Hilbert spaces and for every linear mapping $T$ given on $X_{0}$: if the restriction of $T$ to $X_{j}$ is a bounded operator $T:X_{j}\rightarrow Y_{j}$ for each $j\in\{0,1\}$, then the restriction of $T$ to $X_{\psi}$ is also a bounded operator $T:X_{\psi}\rightarrow Y_{\psi}$.

If $\psi$ is an interpolation parameter, we will say that the Hilbert space $X_{\psi}$ is obtained by the quadratic interpolation with the function parameter $\psi$ between  $X_{0}$ and $X_{1}$ (or of the pair $X$). In this case, we have
\begin{equation}\label{3.18}
\mbox{the dense continuous embeddings}\quad X_{1}\hookrightarrow X_{\psi}\hookrightarrow X_{0}.
\end{equation}

The function $\psi\in\mathcal{B}$ is an interpolation parameter if and only if $\psi$ is pseudoconcave in a neighbourhood of infinity. The latter property means that there exists a concave
function $\psi_{1}:(b,\infty)\rightarrow(0,\infty)$, with $b\gg1$, that both
functions $\psi/\psi_{1}$ and $\psi_{1}/\psi$ are bounded on $(b,\infty)$. This criterion follows from Peetre's \cite{Peetre66, Peetre68} description of all interpolation functions for the weighted Lebesgue spaces (see \cite[Theorem~1.9]{MikhailetsMurach14}).

\begin{proposition}\label{p1}
Let $\varphi\in\mathrm{OR}$, and suppose that real numbers  $s_0<s_1$ satisfy condition \eqref{1.1}. Define the function $\psi$ by the formula
\begin{equation}\label{3.30}
\psi(t):=
\begin{cases}
\;t^{{-s_0}/{(s_1-s_0)}}\,
\varphi(t^{1/{(s_1-s_0)}})&\text{if}\quad t\geq1, \\
\;\varphi(1)&\text{if}\quad0<t<1.
\end{cases}
\end{equation}
Then the function $\psi\in\mathcal{B}$ is an interpolation parameter, and
\begin{equation}\label{2.22}
H^{\varphi}(\mathbb{R}^n)=
\bigl[H^{(s_0)}(\mathbb{R}^n),H^{(s_1)}(\mathbb{R}^n)\bigr]_{\psi}
\end{equation}
with equality of norms.
\end{proposition}

This is a consequence of \cite[Theorem~2.5]{MikhailetsMurachZinchenko25AdvMath} because $H^{\varphi}(\mathbb{R}^n)$ coincides (with equality of norms) with the space $H^{\varphi}_{A}$ of the extended Hilbert scale generated by the operator $(1-\Delta)^{1/2}$ in the Hilbert space $L_{2}(\mathbb{R}^{n})$. (As usual, $\Delta$ is the Laplace operator.)

\begin{remark}
As is indicated in \cite[Remark 2.5]{MikhailetsMurach15ResMath1}, condition \eqref{1.1} is equivalent to the following pair of conditions:
\begin{enumerate}
  \item[(i)] $s_0\leq\sigma_0(\varphi)$ and, moreover, $s_0<\sigma_0(\varphi)$ if the supremum in \eqref{1.2} is not attained;
  \item[(ii)]  $\sigma_1(\varphi)\leq s_1$ and, moreover, $\sigma_1(\varphi)< s_1$ if the infimum in \eqref{1.3} is not attained.
\end{enumerate}
\end{remark}

If $s_{0}<\sigma_{0}(\varphi)$ and $s_{1}>\sigma_{1}(\varphi)$, then Proposition~\ref{p1} is due to \cite[Theorem~2.19]{MikhailetsMurach14}.

\section{Basic results}\label{sec4}

They concerns the properties of the scale~\eqref{1.36}. A version of Proposition~\ref{p1} for this scale is formulated as follows:

\begin{theorem}\label{t4.1}
Let $\varphi\in\mathrm{OR}$, and suppose that real numbers  $s_0<s_1$ satisfy condition \eqref{1.1}. Define the interpolation parameter $\psi$ by formula~\eqref{3.30}.  Then
\begin{equation}\label{3.1}
H^{\varphi}(\Gamma, V) = \bigl[H^{(s_0)}(\Gamma, V),H^{(s_1)}(\Gamma, V)\bigr]_{\psi}
\end{equation}
with equivalence of norms.
\end{theorem}

In the case where $s_{0}<\sigma_{0}(\varphi)$ and $s_{1}>\sigma_{1}(\varphi)$, this Theorem~\ref{t4.1} is proved in \cite[Section~4]{Zinchenko17Collection3}.

\begin{theorem}\label{t4.2a}
For every $\varphi\in\mathrm{OR}$, the space $H^{\varphi}(\Gamma,V)$ is complete (i.e. Hilbert) and separable, and the set $C^\infty(\Gamma,V)$ is dense in $H^{\varphi}(\Gamma,V)$.
\end{theorem}

As is shown in \cite[Section~4]{Zinchenko17Collection3}, we can infer the following result from Theorem~\ref{t4.1}:

\begin{theorem}\label{t4.2b}
The space $H^{\varphi}(\Gamma, V)$ does not depend up to equivalence of norms on the choice of the atlas $\{\alpha_j\}$ and partition of unity $\{\chi_j\}$ on $\Gamma$ and on the choice of the local trivializations $\{\beta_j\}$ of the vector bundle $\pi: \Gamma \rightarrow V$.
\end{theorem}

The extended Sobolev scale is partially ordered with the relation "$\hookrightarrow$" of continuous embedding.

\begin{theorem}\label{t4.3}
Let $\varphi_1, \varphi_2 \in\mathrm{OR}$. The embedding $H^{\varphi_2}(\Gamma, V)\hookrightarrow H^{\varphi_1}(\Gamma, V)$ holds true if and only if the function $\varphi_1/\varphi_2$ is bounded in a neighbourhood of infinity. This embedding is continuous and dense. It is compact  if and only if the function $\varphi_1(t)/\varphi_2(t)\rightarrow0$ as $t\rightarrow\infty$.
\end{theorem}

The extended Sobolev scale is closed with respect to the quadratic interpolation with function parameter. Namely, the following result is true:

\begin{theorem}\label{t4.4}
Let $\varphi_1, \varphi_2 \in\mathrm{OR}$, and let $\psi$ be an interpolation parameter. Suppose that the function $\varphi_{1}/\varphi_{2}$ is bounded in a neighbourhood of infinity, and put
\begin{equation*}
\varphi(t) :=\varphi_1(t)\,\psi\Bigl(\frac{\varphi_2(t)}{\varphi_1(t)}\Bigr) \quad\mbox{whenever}\quad t\geq1.
\end{equation*}
Then $\varphi\in\mathrm{OR}$, and
\begin{equation}\label{3.111111}
\bigl[H^{\varphi_1}(\Gamma, V),H^{\varphi_2}(\Gamma, V)\bigr]_{\psi}=H^{\varphi}(\Gamma, V)
\end{equation}
with equivalence of norms.
\end{theorem}

The extended Sobolev scale consists of all Hilbert spaces that are interpolation spaces between the Sobolev spaces
$H^{(s_0)}(\Gamma,V)$ and $H^{(s_1)}(\Gamma,V)$ whenever $-\infty<s_0<s_1<\infty$.
In this connection, let us recall the definition of an interpolation space between Hilbert spaces $X_0$ and $X_1$. Suppose that
the continuous embedding $X_1\hookrightarrow X_0$ holds.

A Hilbert space $H$ is called an interpolation space between the spaces $X_0$ and $X_1$ if the following two conditions are satisfied:
\begin{enumerate}
  \item[(i)] the continuous embeddings $X_1\hookrightarrow H \hookrightarrow X_0$ are fulfilled;
  \item[(ii)]  for every linear operator $T$ given on $X_0$, the following implication holds: if the restriction of $T$ to $X_j$ is a bounded operator on $X_j$ for each $j\in\{0,1\}$, then the restriction of $T$ to $H$ is a bounded operator on $H$.
\end{enumerate}

Since properties (i) and (ii) are invariant relative to the choice of an equivalent norm on $H$, it is naturally to describe   interpolation spaces up to equivalence of norms.

\begin{theorem}\label{t4.5}
Let $s_0, s_1 \in\mathbb{R}$ and $s_0<s_1$.
A Hilbert space $H$ is an interpolation space between the Sobolev spaces $H^{(s_0)}(\Gamma,V)$ and $H^{(s_1)}(\Gamma,V)$ if and only if $H = H^\varphi(\Gamma,V)$ up to equivalence of norms for a certain parameter $\varphi\in\mathrm{OR}$ that satisfies condition \eqref{1.1}.
\end{theorem}

A Hilbert space $H$ is called an interpolation space with respect to the Sobolev scale $\{H^{(s)}(\Gamma,V):s\in\mathbb{R}\}$ if $H$ is interpolation space between certain spaces belonging to this scale. Owing to Theorem \ref{t4.5}, we have

\begin{corollary}\label{c1}
A Hilbert space $H$ is an interpolation space with respect to the Sobolev scale $\{H^{{s}}(\Gamma,V): s\in \mathbb{R})\}$ if and only if  $H=H^\varphi(\Gamma,V)$ up to equivalence of norms for some $\varphi\in\mathrm{OR}$.
\end{corollary}

Suppose now that the vector bundle $\pi: V \rightarrow\Gamma$ is Hermitian. Thus, for every $x\in\Gamma$, a certain inner product $\langle \cdot, \cdot \rangle_{x}$ is defined in the fiber $\pi^{-1}(x)$ so that the scalar function $\Gamma\ni x\mapsto\langle u(x), v(x)\rangle_{x}$ is infinitely smooth on $\Gamma$ for arbitrary sections $u,v\in C^\infty(\Gamma, V)$. Using the $C^\infty$-density $dx$ on $\Gamma$, we define the inner product of these sections by the formula
\begin{equation}\label{4.332}
\langle u,v\rangle_{\Gamma, V} := \int\limits_{\Gamma} \langle u(x),v(x)\rangle_{x}\,dx.
\end{equation}

\begin{theorem}\label{t4.6}
For every $\varphi\in\mathrm{OR}$, the spaces $H^{\varphi}(\Gamma, V)$ and $H^{1/\varphi}(\Gamma, V)$ are mutually dual (up to equivalence of norms)  with respect to the sesquilinear form \eqref{4.332}.
\end{theorem}

Note that $\varphi\in\mathrm{OR}\Leftrightarrow 1/\varphi\in\mathrm{OR}$; hence, the space $H^{1/\varphi}(\Gamma, V)$ is well defined in this theorem.

Ending this section, we consider a relation between the extended Sobolev scale and the scale $\{C^{q}(\Gamma, V): 0\leq q\in\mathbb{Z}\}$. As usual, $C^{q}(\Gamma, V)$ denotes the Banach space of all $q$ times continuously differentiable sections $u:\Gamma\rightarrow V$.

\begin{theorem}\label{t4.7}
Let $\varphi\in\mathrm{OR}$ and $0\leq q\in\mathbb{Z}$. Then the condition
\begin{equation}\label{4.15}
\int\limits_1^\infty t^{2q+n-1}\varphi^{-2}(t)\,dt<\infty
\end{equation}
is equivalent to the embedding $H^{\varphi}(\Gamma, V) \subset C^q(\Gamma, V)$. Moreover, this embedding is compact.
\end{theorem}

\section{Auxiliary results}

For the readers' convenience, we formulate properties of the quadratic interpolation which will be used in our proofs, we following the monograph \cite[Section~1.1]{MikhailetsMurach14}. The first of these properties concerns the embedding of spaces obtained by the interpolation \cite[Theorem~1.2]{MikhailetsMurach14}.

\begin{proposition}\label{p6}
Let functions $\psi_1, \psi_2\in\mathcal{B}$, and suppose that the function $\psi_1/\psi_2$ is bounded in a neighbourhood  of infinity. Then, for every regular pair $X = [X_0, X_1]$ of Hilbert spaces, we have the dense continuous embedding $X_{\psi_2} \hookrightarrow X_{\psi_1}$. If the embedding $X_1\hookrightarrow X_0$ is compact and if $\psi_1(t)/\psi_2(t)\rightarrow0$ as $t\rightarrow\infty$, then the embedding $X_{\psi_2} \hookrightarrow X_{\psi_1}$ is also compact.
 \end{proposition}

The next property reduces the interpolation between orthogonal sums of Hilbert spaces to the interpolation between the summands \cite[Theorem~1.5]{MikhailetsMurach14}.

\begin{proposition}\label{p2}
Let $\bigl[X_{0}^{(j)},X_{1}^{(j)}\bigr]$, with $j=1,\ldots,r$, be a finite collection of regular couples of Hilbert spaces. Then, for every function $\psi\in\mathcal{B}$, we have
\begin{equation*}
\biggl[\,\bigoplus_{j=1}^{r}X_{0}^{(j)},\,\bigoplus_{j=1}^{r}X_{1}^{(j)}\biggr]_{\psi}=\,
\bigoplus_{j=1}^{r}\bigl[X_{0}^{(j)},\,X_{1}^{(j)}\bigr]_{\psi}
\end{equation*}
with equality of norms.
\end{proposition}

The interpolation possesses the following reiteration property \cite[Theorem~1.3]{MikhailetsMurach14}:

\begin{proposition}\label{p7}
Suppose that $\psi_1,\psi_2,\psi\in\mathcal{B}$ and that the function $\psi_1/\psi_2$ is bounded in a neighbourhood  of infinity. Then, for every regular pair $X$ of Hilbert spaces, we have $[X_{\psi_1}, X_{\psi_2}]_\psi = X_\omega$ with equality of norms. Here, the function $\omega(t):= \psi_1(t)\psi(\psi_2(t)/\psi_1(t))$ of $t>0$ belongs to $\mathcal{B}$. If the functions $\psi_1,\psi_2,\psi$ are interpolation  parameters, then $\omega$ is also an interpolation  parameter.
\end{proposition}

The last property reduces the interpolation between the dual or antidual spaces of given Hilbert spaces to the interpolation between these given spaces \cite[Theorem~1.4]{MikhailetsMurach14}. We need this property in the case of antidual spaces. If $H$ is a Hilbert space, then we let $H'$ denote the antidual of $H$; namely, $H'$ consists of all antilinear continuous functionals $l: H \rightarrow \mathbb{C}$. The linear space $H'$ is Hilbert with respect to the inner product  $(l_1, l_2)_{H'} := (v_1, v_2)_{H}$ of functionals $l_1, l_2\in H'$; here $v_j$, with $j\in \{1, 2\}$, is a unique vector from $H$ such that $l_j(w) = (v_j, w)_H$ for every $w\in H$. Note that we do not identify $H$ and $H'$ on the base of the Riesz theorem (according to which  $v_j$ exists).

\begin{proposition}\label{p5}
Let a function $\psi\in\mathcal{B}$ be such that the function $\psi(t)/t$ is bounded in a neighbourhood  of infinity. Then, for every regular pair $[X_0, X_1]$ of Hilbert spaces, we have $[X_1',X_0']_{\psi}=[X_0,X_1]'_{\chi}$ with equality of norms. Here, the function $\chi\in\mathcal{B}$ is defined by the formula $\chi(t):=t/\psi(t)$ for $t>0$. If $\psi$ is an interpolation parameter, then $\chi$ is an interpolation parameter as well.
\end{proposition}

As to this theorem, we note that, if $[X_0, X_1]$ is a regular pair of Hilbert spaces, then the dual pair $[X_1', X_0']$ is also regular provided that we identify functionals from $X_0'$ with their restrictions on $X_1$.

We also need the following corollary of Ovchinnikov's theorem \cite[Theorem 11.4.1]{Ovchinnikov}:

\begin{proposition}\label{p4}
Let $[X_0, X_1]$ be a regular pair of Hilbert spaces. A Hilbert space  $H$ is an interpolation space between $X_0$ and $X_1$ if and only if $H=X_\psi$ up to equivalence of norms for a certain interpolation parameter $\psi\in\mathcal{B}$.
\end{proposition}

\section{Proofs of the basic results}\label{sec7}

\begin{proof}[Proof of Theorem $\ref{t4.1}$]
Mainly following \cite[Section~4]{Zinchenko17Collection3}, we will deduce the required equality \eqref{3.1} from Proposition~\ref{p1} with the help of certain operators of flattening  and sewing of the vector bundle $\pi: V \rightarrow\Gamma$. Note that the pair of Sobolev spaces on the right of \eqref{3.1} is regular (see, e.g., \cite[Chapter~4, Proposition~1.2]{Wells}).

Define the flattening operator by the formula
\begin{equation*}
T: u\mapsto ((\chi_1 u^{\times}_{1})\circ\alpha_1,\ldots,
(\chi_\varkappa u^{\times}_{\varkappa})\circ\alpha_\varkappa)
\end{equation*}
for arbitrary $u\in \mathcal{D}'(\Gamma, V)$. Recall that
$u^{\times}_j$ stands for the representation of the generalized section $u$ in the local trivialization  $\beta_j$.
The mapping $T$ sets isometrical operators
\begin{equation}\label{3.3}
T: H^{\varphi}(\Gamma, V)\rightarrow (H^{\varphi}(\mathbb{R}^n))^{p\varkappa}
\end{equation}
and
\begin{equation*}
T: H^{(s)}(\Gamma, V)\rightarrow (H^{(s)}(\mathbb{R}^n))^{p\varkappa} \quad \mbox{for every} \quad s \in \mathbb{R}.
\end{equation*}
The latter implies that $T$ is a bounded operator
\begin{equation}\label{3.4}
T: \bigl[H^{(s_0)}(\Gamma, V),H^{(s_1)}(\Gamma, V)\bigr]_{\psi}\rightarrow \bigl[(H^{(s_0)}(\mathbb{R}^n))^{p\varkappa}, (H^{(s_1)}(\mathbb{R}^n))^{p\varkappa}\bigr]_{\psi}.
\end{equation}
Owing to Propositions~\ref{p1} and \ref{p2}, we have
\begin{equation}\label{3.5}
\begin{aligned}
\bigl[(H^{(s_0)}(\mathbb{R}^n))^{p\varkappa}, (H^{(s_1)}(\mathbb{R}^n))^{p\varkappa}\bigr]_{\psi} & = \bigl([H^{(s_0)}(\mathbb{R}^n),H^{(s_1)}(\mathbb{R}^n)]_{\psi}\bigr)^{p\varkappa} \\
&= \bigl(H^{\varphi}(\mathbb{R}^n)\bigr)^{p\varkappa}.
\end{aligned}
\end{equation}
Hence, the bounded operator \eqref{3.4} acts between the spaces \begin{equation}\label{3.6}
T:[H^{(s_0)}(\Gamma, V),H^{(s_1)}(\Gamma, V)]_{\psi}\rightarrow \bigl(H^{\varphi}(\mathbb{R}^n)\bigr)^{p\varkappa}.
\end{equation}

Define the sewing operator by the formula
\begin{equation}\label{3.7}
K: \mathbf{w} \mapsto \sum_{j=1}^\varkappa \Theta_j ((\eta_j w_j)\circ\alpha_j^{-1})
\end{equation}
for every $\mathbf{w} := (w_{1}, \ldots, w_{\varkappa}) \in (\mathcal{S}'(\mathbb{R}^n))^{p\varkappa}$. Here, for each $j\in\{1, \ldots, \varkappa\}$, the function $\eta_j \in C^\infty_0(\mathbb{R}^n)$ is chosen so that $\eta_j = 1$ in a neighbourhood of $\alpha^{-1}_j(\mathrm{supp} \chi_j)$. (As usual, $C^\infty_0(\mathbb{R}^n)$ is the space of all compactly supported $C^\infty$-functions on $\mathbb{R}^n$.) Moreover, given $\omega\in(\mathcal{D}'(\Gamma_j))^p$ subject to $\mathrm{dist}(\mathrm{supp}\,\omega,\partial\Gamma_j)>0$,
we let $\Theta_j\omega$ denote the unique generalized section of the vector bundle $\pi:V\to\Gamma$ such that $(\Theta_j\omega)^{\times}_{j}=\omega$ on $\Gamma_{j}$ and that $(\Theta_j\omega)^{\times}_{j}=0$ on $\Gamma\setminus\overline{\Gamma}_{j}$, with $(\Theta_j\omega)^{\times}_{j}$ standing for the representation of $\Theta_j\omega$ in the local trivialization $\beta_{j}$. We take $\omega:=(\eta_j w_j)\circ\alpha_j^{-1}$ in \eqref{3.7}.

The operator $K$ is left inverse to $T$. Indeed, given $u\in \mathcal{D}'(\Gamma, V)$, we have the following equalities:
\begin{align*}
KTu&=\sum_{j=1}^\varkappa \Theta_j\Bigl(\bigl(\eta_j ((\chi_j u^{\times}_{j})\circ\alpha_j)\bigr)\circ\alpha_j^{-1}\Bigr)\\
&= \sum_{j=1}^\varkappa \Theta_j\Bigl(\bigl(\eta_j\circ\alpha_j^{-1}\bigr) (\chi_j u^{\times}_{j})\Bigr)=\sum_{j=1}^\varkappa \Theta_j (\chi_j u^{\times}_{j})=
\sum_{j=1}^\varkappa \chi_j u = u.
\end{align*}

Let us show that $K$ is a bounded operator
\begin{equation}\label{3.14}
K: (H^\varphi(\mathbb{R}^n))^{p\varkappa}\to H^\varphi(\Gamma,V).
\end{equation}
Considering arbitrary $\mathbf{w}:=(w_{1},\ldots,w_{\varkappa})\in (H^\varphi(\mathbb{R}^n))^{p\varkappa}$ and using the definition of the norm in $H^\varphi(\Gamma, V)$, we have the equalities:
\begin{align*}
\| K\mathbf{w} \|_{H^\varphi(\Gamma,V)}^2 &=
\sum_{j=1}^\varkappa \|\bigl(\chi_j(K\mathbf{w})^{\times}_j\bigr)\circ\alpha_j\|_
{(H^\varphi(\mathbb{R}^n))^p}^2 \\
&= \sum_{j=1}^\varkappa\,\Bigl\|\sum_{l=1}^\varkappa \Bigl(\chi_j\bigl(\Theta_l((\eta_lw_l)\circ
\alpha_l^{-1})\bigr)^{\times}_j\Bigr)
\circ\alpha_j\Bigr\|_{(H^\varphi(\mathbb{R}^n))^p}^2,
\end{align*}
with $(\cdot)^{\times}_j$ denoting the representation of the generalized section (written between the parentheses) in the local trivialization $\beta_j$.
Here,
\begin{align*}
\Bigl(\chi_j\bigl(\Theta_l((\eta_lw_l)\circ
\alpha_l^{-1})\bigr)^{\times}_j\Bigr)
\circ\alpha_j & =
\Bigl(\Theta_{j,l}\bigl(\chi_j\beta_{j,l}((\eta_lw_l)\circ
\alpha_l^{-1})\bigr)\Bigr)\circ\alpha_j \\
& = (\eta_{l,j} w_l)\circ\gamma_{l,j},
\end{align*}
where $\Theta_{j,l}$ stands for the operator  of the extension by zero from $\Gamma_{j}\cap\Gamma_{l}$ to $\Gamma$ of a vector (or matrix)-valued distribution which is defined on $\Gamma_{j}\cap\Gamma_{l}$ and whose support does not abut on $\partial(\Gamma_{j}\cap\Gamma_{l})$, whereas the matrix-valued function $\beta_{l,j}\in C^\infty(\Gamma_{j}\cap\Gamma_{l}, \mathbb{C}^{p\times p})$ is defined by the formula $(x,\beta_{l,j}(x)a):=(\beta_l\circ\beta_j^{-1})(x,a)$ for arbitrary $x\in \Gamma_{j}\cap\Gamma_{l}$ and $a\in \mathbb{C}^p$. Besides,
$$
\eta_{l,j}:=\bigl(\Theta_{j,l}((\chi_j\beta_{j,l})
(\eta_l\circ\alpha_l^{-1}))\bigr)\circ\alpha_{l}
\in C^\infty_0(\mathbb{R}^n, \mathbb{C}^{p\times p}),
$$
and $\gamma_{l,j}:= \mathbb{R}^n \rightarrow \mathbb{R}^n$ is an infinitely smooth diffeomorphism such that $\gamma_{l,j}:= \alpha_l^{-1}\circ\alpha_j$ in a neighbourhood of $\mathrm{supp}\,\eta_{l,j}$ and that $\gamma_{l,j}(t)= t$ for every $t\in \mathbb{R}^n$ subject to $|t|\gg1$. Thus,
\begin{align*}
\|K\mathbf{w}\|_{{H^\varphi(\Gamma,V)}}^2 &=
\sum_{j=1}^\varkappa\,\Bigl\|\sum_{l=1}^\varkappa
(\eta_{l,j} w_l)\circ\gamma_{l,j} \Bigr\|_{(H^\varphi(\mathbb{R}^n))^p}^2\\
&\leq c \sum_{l=1}^\varkappa\|w_l\|_{(H^\varphi(\mathbb{R}^n))^p}^2=
c\,\|\mathbf{w}\|_{(H^\varphi(\mathbb{R}^n))^{p\varkappa}}^2\,,
\end{align*}
where $c$ is a certain positive number which does not depend on $\mathbf{w}$.
The last inequality is a consequence of the fact, that the operator of the multiplication by a function from  $C^{\infty}_{0}(\mathbb{R}^n)$ and the operator $v\mapsto v\circ\gamma_{l,j}$ of $C^{\infty}$-change of variables are bounded on the space $H^\varphi(\mathbb{R}^n)$. Such properties of this space follow from their known analogues for Sobolev spaces over $\mathbb{R}^n$ in view of Proposition~1. Thus, $K$ is a bounded operator between  spaces~\eqref{3.14}.

Specifically, $K$ acts continuously on the pair of spaces
\begin{equation*}
K: ((H^{(s)}(\mathbb{R}^n))^{p\varkappa} \rightarrow H^{(s)}(\Gamma, V) \quad \mbox{for every}\quad s \in \mathbb{R}.
\end{equation*}
It follows from this by Propositions~\ref{p1} and formula \eqref{3.5} that the operator $K$ is also bounded between the spaces
\begin{equation}\label{3.15}
K: (H^{\varphi}(\mathbb{R}^n))^{p\varkappa}  \rightarrow [H^{(s_0)}(\Gamma, V),H^{(s_1)}(\Gamma, V)]_{\psi}.
\end{equation}

Now, the continuous embedding
\begin{equation*}
H^\varphi(\Gamma, V)\hookrightarrow
[H^{(s_0)}(\Gamma, V),H^{(s_1)}(\Gamma, V)]_{\psi}.
\end{equation*}
is the product of the bounded operators \eqref{3.15} and \eqref{3.3}, with the inverse being the product of the bounded operators \eqref{3.14} and \eqref{3.6}.
\end{proof}


\begin{proof}[Proof of Theorem $\ref{t4.2a}$]
This theorem is known in the Sobolev case where $\varphi(t)\equiv t^{s}$ for some $s\in\mathbb{R}$ (see, e.g, \cite[Chapter~IV, Section~1]{Wells}). In the general situation, this theorem follows from the Sobolev case by virtue of Theorem $\ref{t4.1}$. Namely, according to this theorem, the space $H^{\varphi}(\Gamma, V)$ is complete (i.e. Hilbert) and separable due to the properties of the interpolation with a function parameter mentioned in Section \ref{sec3}. Besides, owing to \eqref{3.18},
we have the continuous and dense embedding $H^{(s_1)}(\Gamma, V)\hookrightarrow H^{\varphi}(\Gamma, V)$ if $s_1>\sigma_{0}(\varphi)$. Since the set $C^{\infty}(\Gamma, V)$ is dense in $H^{(s_1)}(\Gamma, V)$, it is also dense in  $H^{\varphi}(\Gamma, V)$ due to this embedding.
\end{proof}

\begin{proof}[Proof of Theorem $\ref{t4.2b}$]
Consider two triplets $\mathcal{A}_{1}$ and $\mathcal{A}_{2}$ each of which is formed by an atlas of the manifold $\Gamma$, appropriate partition of unity on $\Gamma$, and collection of local trivializations of the vector bundle  $\pi: V\rightarrow\Gamma$. Let $H^{\varphi}(\Gamma, V; \mathcal{A}_{j})$ and $H^{(s)}(\Gamma, V; \mathcal{A}_{j})$, with  $s\in\mathbb{R}$, respectively denote the H\"ormander space $H^{\varphi}(\Gamma, V)$ and the Sobolev space $H^{(s)}(\Gamma, V)$ corresponding to the triplet $\mathcal{A}_{j}$ with $j\in\{1,\,2\}$.
The conclusion of Theorem~\ref{t4.2b} holds true in  the Sobolev case of $\varphi\equiv t^s$ (see, e.g., \cite[Chapter~IV, Section~1, p.~110]{Wells}); i.e., the identity mapping sets an isomorphism between the spaces $H^{(s)}(\Gamma,V;\mathcal{A}_{1})$ and $H^{(s)}(\Gamma, V; \mathcal{A}_{2})$ for every $s\in\mathbb{R}$. It follows from this by Theorem~\ref{t4.1} that the identity mapping also sets an isomorphism between the spaces
\begin{equation*}
\bigl[H^{(s_0)}(\Gamma,V;\mathcal{A}_{j}),
H^{(s_1)}(\Gamma, V;\mathcal{A}_{j})\bigr]_\psi=
H^{\varphi}(\Gamma, V;\mathcal{A}_{j})
\end{equation*}
with $j=0$ and $j=1$; here, $s_0$, $s_1$, and $\psi$ satisfy the hypotheses of this theorem.
\end{proof}

\begin{proof}[Proof of Theorem~$\ref{t4.3}$]
We choose numbers $s_0, s_1 \in \mathbb{R}$ such that $s_0 <\min\{\sigma_0(\varphi_1), \sigma_0(\varphi_2)\}$ and $ s_1> \max\{\sigma_1(\varphi_1), \sigma_1(\varphi_2)\}$. According to Theorem~\ref{t4.1},
\begin{equation}\label{3.16}
[H^{(s_0)}(\Gamma, V),H^{(s_1)}(\Gamma, V)]_{\psi_j}= H^{\varphi_j}(\Gamma, V)\quad\mbox{for each}\quad j\in\{1, 2\}
\end{equation}
up to equivalence of norms. Here, $\psi_j$ is the interpolation parameter defined by formula \eqref{3.30} in which we take $\varphi := \varphi_j$ and $\psi_j:=\psi$.

If the function $\varphi_1/\varphi_2$ is bounded in a neighbourhood of infinity, then the function $\psi_1(t)/\psi_2(t) = \varphi_1(t^{1/(s_1-s_0)})/\varphi_2(t^{1/(s_1-s_0)})$ of $t\geq 1$ is also bounded there, and then the dense continuous embedding $H^{\varphi_2}(\Gamma, V) \hookrightarrow H^{\varphi_1}(\Gamma, V)$ holds due to Proposition~\ref{p6}. If $\varphi_1(t)/\varphi_2(t)\rightarrow0$ as $t\rightarrow\infty$, then $\psi_1(t)/\psi_2(t) \rightarrow 0$ as $t\rightarrow\infty$, and then this embedding is compact due to the same proposition and the known compact embedding $H^{(s_1)}(\Gamma, V) \hookrightarrow H^{(s_0)}(\Gamma, V)$ (see, e.g., \cite[Chapter~4, Section~1, Proposition~1.2]{Wells}).

Assume now that $H^{\varphi_2}(\Gamma, V) \subset H^{\varphi_1}(\Gamma, V)$, and prove that the function $\varphi_1/\varphi_2$ is bounded on $[1,\infty)$. Without loss of generality we suppose that $\Gamma_1\not\subset \cup_{j=2}^\varkappa \Gamma_j$. Choose an open nonempty set  $U\subset\Gamma_1$ such that $U\cap\Gamma_j=\emptyset$ whenever $j\neq1$. Consider an arbitrary distribution $\omega\in H^{\varphi_2}(\mathbb{R}^n)$ such that $\mathrm{supp}\,\omega\subset\alpha^{-1}_1(U)$.
According to \eqref{3.14} and our assumption, we have the inclusion
\begin{equation*}
u:=K(\omega, \underbrace{0, \ldots, 0}_{p\varkappa-1}) \in H^{\varphi_2}(\Gamma, V) \subset H^{\varphi_1}(\Gamma, V);
\end{equation*}
here, $K$ is the sewing operator from the proof of Theorem~\ref{t4.1}.
Therefore
\begin{equation*}
(\omega, \underbrace{0, \ldots, 0}_{p-1}) = (\chi_1 u_1)\circ\alpha_1 \in (H^{\varphi_1}(\mathbb{R}^n))^p
\end{equation*}
due to the definition of $H^{\varphi_1}(\Gamma, V)$. (As to the latter equality, note that $\chi_1 u_1 = u_1$ because $\chi_1 = 1$ on $U$). Thus,
\begin{equation*}
\bigl\{\omega\in H^{\varphi_2}(\mathbb{R}^n): \mathrm{supp}\,\omega\subset \alpha_1^{-1}(U)\bigr\} \subset H^{\varphi_1}(\mathbb{R}^n).
\end{equation*}
It follows from this by \cite[Theorem 2.2.2]{Hermander63} that the function $\varphi_1(\langle\xi\rangle)/\varphi_2(\langle\xi\rangle)$ of $\xi\in \mathbb{R}^n$ is bounded on $\mathbb{R}^n$. Therefore, the function $\varphi_1/\varphi_2$ is bounded on $[1,\infty)$.

Assume in addition that the embedding $H^{\varphi_2}(\Gamma, V) \hookrightarrow H^{\varphi_1}(\Gamma, V)$ is compact, and prove that $\varphi_1(t)/\varphi_2(t)\rightarrow0$ as $t\rightarrow\infty$.
Choose a closed ball $Q\subset \alpha^{-1}_1(U)$, and put
\begin{equation*}
H^{\varphi_2}_Q(\mathbb{R}^n):= \bigl\{\omega\in H^{\varphi_2}(\mathbb{R}^n): \mathrm{supp}\,\omega\subset Q\bigr\}.
\end{equation*}
We consider $H^{\varphi_2}_Q(\mathbb{R}^n)$ as a (closed) subspace of $H^{\varphi_2}(\mathbb{R}^n)$. Let us represent the embedding  $H^{\varphi_2}_Q(\mathbb{R}^n) \hookrightarrow H^{\varphi_1}(\mathbb{R}^n)$ as the composition of the following three continuous operators:
\begin{equation*}
\omega\mapsto u:= K(\omega, \underbrace{0, \ldots, 0}_{p\varkappa-1}) \mapsto u \mapsto (\chi_1 u_1)\circ\alpha_1 =
(\omega, \underbrace{0, \ldots, 0}_{p-1}).
\end{equation*}
The first of them acts from $H^{\varphi_2}_Q(\mathbb{R}^n)$ to $H^{\varphi_2}(\Gamma, V)$ by \eqref{3.14}; the second is the operator of the compact embedding of $H^{\varphi_2}(\Gamma, V)$ in $H^{\varphi_1}(\Gamma, V)$, and the third acts from  $H^{\varphi_1}(\Gamma, V)$ to $H^{\varphi_1}(\mathbb{R}^n)$. Thus, the embedding $H^{\varphi_2}_Q(\mathbb{R}^n) \hookrightarrow H^{\varphi_1}(\mathbb{R}^n)$ is compact.
It follows from this by \cite[Theorem 2.2.3]{Hermander63} that $\varphi_1(\langle\xi\rangle)/\varphi_2(\langle\xi\rangle)\rightarrow0$ as $|\xi|\rightarrow\infty$; i.e., $\varphi_1(t)/\varphi_2(t)\rightarrow0$ as $t\rightarrow\infty$.
\end{proof}

\begin{proof}[Proof of Theorem~$\ref{t4.4}$]
The inclusion $\varphi\in\mathrm{OR}$ is valid due to \cite[Theorem 2.18]{MikhailetsMurach14}.
Choose numbers $s_0, s_1 \in \mathbb{R}$ such that
\begin{equation*}
s_0 <\min\{\sigma_0(\varphi), \sigma_0(\varphi_1), \sigma_0(\varphi_2)\} \quad \mbox{and}\quad
s_1> \max\{\sigma_1(\varphi), \sigma_1(\varphi_1), \sigma_1(\varphi_2)\}.
\end{equation*}
The pair $[H^{\varphi_1} (\Gamma, V), H^{\varphi_2}(\Gamma, V)]$ is regular due to Theorems \ref{t4.2a} and~\ref{t4.3}.
Let us use the interpolation formula \eqref{3.16}, where the function $\psi_1 / \psi_2$ is bounded in a neighbourhood of infinity. Owing to Proposition~\ref{p7}, we have
\begin{align}
&\bigl[H^{\varphi_1} (\Gamma, V), H^{\varphi_2}(\Gamma, V)\bigr]_{\psi}\notag\\
& = \bigl[\bigl[H^{(s_0)} (\Gamma, V), H^{(s_1)}(\Gamma, V)\bigr]_{\psi_1}, \bigl[H^{(s_0)} (\Gamma, V), H^{(s_1)}(\Gamma, V)\bigr]_{\psi_2}\bigr]_{\psi}\notag\\
& = \bigl[H^{(s_0)} (\Gamma, V), H^{(s_1)}(\Gamma, V)\bigr]_{\omega}.\label{3.17}
\end{align}
Here, the interpolation parameter $\omega$ satisfies the equalities
\begin{align*}
\omega(t) & = \psi_1(t)\psi\Bigl(\frac{\psi_2(t)}{\psi_1(t)}\Bigr) = t^{-s_0/(s_1-s_0)}\varphi_1(t^{1/(s_1-s_0)})
\psi\Bigl(\frac{\varphi_2(t^{1/(s_1-s_0)})}{\varphi_1(t^{1/(s_1-s_0)})}\Bigr)\\
& = t^{-s_0/(s_1-s_0)}\varphi (t^{1/(s_1-s_0)})
\end{align*}
whenever $t\geq1$. Besides,
\begin{align*}
\omega(t) & = \psi_1(t)\psi\Bigl(\frac{\psi_2(t)}{\psi_1(t)}\Bigr) = \varphi_1(1)\psi\Bigl(\frac{\varphi_2(1)}{\varphi_1(1)}\Bigr) = \varphi(1)
\end{align*}
whenever $0<t<1$.
Thus, equality \eqref{3.30} holds true if we take $\omega$ instead of $\psi$.
Therefore,
\begin{equation}\label{3.19}
\bigl[H^{(s_0)} (\Gamma, V), H^{(s_1)}(\Gamma, V)\bigr]_{\omega} = H^{\varphi} (\Gamma, V)
\end{equation}
due to Theorem \ref{t4.1}. Relations \eqref{3.17} and \eqref{3.19} yield the required formula \eqref{3.111111}.
\end{proof}

\begin{proof}[Proof of Theorem $\ref{t4.5}$]
According to Proposition \ref{p4}, a Hilbert space $H$ is an interpolation space between the spaces $H^{(s_0)} (\Gamma, V)$ and  $H^{(s_1)}(\Gamma, V)$ if and only if $H$ equals the space $[H^{(s_0)} (\Gamma, V), H^{(s_1)}(\Gamma, V)]_{\psi}$ for a certain interpolation parameter $\psi\in\mathcal{B}$. Owing to Theorem~\ref{t4.4} considered in the case of $\varphi_j(t) \equiv t^{s_j}$, the last space equals
$ H^{\varphi}(\Gamma, V)$ provided that the function  $\varphi\in\mathrm{OR}$ is defined by the formula $\varphi(t) \equiv t^{s_0}\psi(t^{s_1-s_0})$. These equalities of spaces hold up to equivalence of norms. It remains to note in view of \cite[Theorem~4.2]{MikhailetsMurach15ResMath1} that $\varphi$ satisfies \eqref{1.1} if and only if the function $\psi\in\mathcal{B}$ is an interpolation parameter.
\end{proof}

\begin{proof}[Proof of Theorem $\ref{t4.6}$]
This theorem is known in the Sobolev case where $\varphi(t)\equiv t^s$ for some $s\in\mathbb{R}$ (see, e.g., \cite[Chapter~IV, Section~1, p.~110]{Wells}). Thus, the continuous sesquilinear form $u \mapsto\langle u, v\rangle_{\Gamma, V}$ of the arguments $u\in H^{s}(\Gamma, V)$ and $v\in H^{-s}(\Gamma, V)$ is well defined as a unique extension of the form \eqref{4.332}. Moreover, the mapping  $Q: u \mapsto\langle u, \cdot\rangle_{\Gamma, V}$ is an isomorphism between $H^{s}(\Gamma, V)$ and $(H^{-s}(\Gamma, V))'$. Let deduce this theorem from the Sobolev case with the help of Proposition~\ref{p5}.
Choose numbers $s_0, s_1 \in \mathbb{R}$ such that $s_0 <\sigma_0(\varphi)$ and $s_1> \sigma_1(\varphi)$.
Considering this isomorphism for $s \in \{s_0, s_1\}$ and using the interpolation with the function parameter $\psi$ defined by \eqref{3.30}, we conclude by Theorem~\ref{t4.1} that $Q$ sets an isomorphism between spaces \eqref{2.22} and
\begin{equation}\label{7.3}
\bigl[(H^{(-s_0)} (\Gamma, V))', (H^{(-s_1)}(\Gamma, V))'\bigr]_\psi.
\end{equation}
According to Proposition~\ref{p5} and Theorem~\ref{t4.4}, space \eqref{7.3} equals
\begin{equation*}
\bigl[H^{(-s_1)} (\Gamma, V), H^{(-s_0)}(\Gamma, V)\bigr]_\chi'= (H^{1/\varphi}(\Gamma, V))',
\end{equation*}
where $\chi(t)\equiv t/\psi(t)$. We have used Theorem~\ref{t4.4} in the case where $\varphi_1(t)\equiv t^{-s_1}$ and $\varphi_2(t)\equiv t^{-s_0}$; note that
\begin{equation*}
\varphi_1(t)\chi\Bigl(\frac{\varphi_2(t)}{\varphi_1(t)}\Bigr) = t^{-s_1}\chi(t^{s_1-s_0}) = \frac{1}{t^{s_0}\psi(t^{s_1-s_0})} = \frac{1}{\varphi(t)}
\end{equation*}
whenever $t\geq1$, the last equality being valid due to \eqref{3.30}. Thus, $Q$ sets an isomorphism between $H^{\varphi}(\Gamma,V)$ and $(H^{1/\varphi} (\Gamma,V))'$.
\end{proof}

The proof of Theorem \ref{t4.7} is based on the following version of H\"ormander's embedding theorem \cite[Theorem 2.2.7]{Hermander63}:

\begin{proposition}\label{p3}
Let $\varphi\in\mathrm{OR}$. Then condition \eqref{4.15}
implies the continuous embedding $H^{\varphi}(\mathbb{R}^n)\hookrightarrow C^{q}_\mathrm{b}(\mathbb{R}^n)$. Conversely, if
\begin{equation}\label{6.1}
\{w\in H^{\varphi}(\mathbb{R}^n): \mathrm{supp}\,w\subset G\} \subset C^q(\mathbb{R}^n)
\end{equation}
for a certain open nonempty set $G\subset\mathbb{R}^n$, then condition~\eqref{4.15} is satisfied.
\end{proposition}

Here, $C^{q}_\mathrm{b}(\mathbb{R}^n)$ denotes the Banach space of all $q$ times continuously differentiable functions on $\mathbb{R}^n$ whose partial derivatives up to the $q$-th order are bounded on $\mathbb{R}^n$. Note that the condition
\begin{equation*}
\int \limits_{\mathbb{R}^n}\frac{(1+|\xi|)^{2q}}{k^2(\xi)} \, d\xi<\infty
\end{equation*}
used by H\"ormander for the space $\mathcal{B}_{2,k}$ is equivalent to \eqref{4.15} in our case where $k(\xi)\equiv\varphi(\langle\xi\rangle)$ and, hence, $\mathcal{B}_{2,k}=H^\varphi(\mathbb{R}^n)$. This equivalence is proved in  \cite[Lema~2]{ZinchenkoMurach12UMJ11}.

\begin{proof}[Proof of Theorem $\ref{t4.7}$]
Suppose that condition~\eqref{4.15} is satisfied. Then we have the continuous embedding $H^\varphi(\mathbb{R}^n)\hookrightarrow C^q_b(\mathbb{R}^n)$ due to Proposition~\ref{p3}.
Choosing $u\in H^\varphi(\Gamma, V)$ arbitrarily, we obtain the inclusion
\begin{equation*}
(\chi_ju_j)\circ\alpha_j\in (H^\varphi(\mathbb{R}^n))^p\hookrightarrow(C^q_b(\mathbb{R}^n))^p
\end{equation*}
for each $j\in\{1,\ldots, \varkappa\}$. Hence, each $\chi_ju\in C^q(\Gamma, V)$, which implies the inclusion
\begin{equation*}
u = \sum_{j=1}^\varkappa \chi_j u \in C^q(\Gamma, V).
\end{equation*}
Thus, $H^\varphi(\Gamma, V)\subset C^q(\Gamma,V)$. This embedding is continuous because
\begin{equation*}
\|u\|_{C^q(\Gamma,V)}\leq c_0 \sum_{j=1}^\varkappa \|(\chi_ju_j)\circ\alpha_j\|_{(C^q_b(\mathbb{R}^n))^p} \leq c_1 \sum_{j=1}^\varkappa \|(\chi_ju_j)\circ\alpha_j\|_{(H^\varphi(\mathbb{R}^n))^p} \leq c_2 \|u\|_{H^\varphi(\Gamma, V)}.
\end{equation*}
Here, $c_0$, $c_1$, and $c_2$ are certain positive numbers that do not depend on $u$.

Let us prove that this embedding is compact. In the next paragraph, we will build a function $\varphi_0\in \mathrm{OR}$ such that $\varphi_0(t)/\varphi(t)\rightarrow0$ as $t\rightarrow\infty$ and that
\begin{equation}\label{4.15-0}
\int\limits_1^\infty \frac{t^{2q+n-1}}{\varphi_0^2(t)} \, dt<\infty.
\end{equation}
Owing to Theorem \ref{t4.3}, we have the compact embedding $H^\varphi(\Gamma, V)\hookrightarrow H^{\varphi_0}(\Gamma, V)$. Besides, the continuous embedding $H^{\varphi_0}(\Gamma, V)\hookrightarrow C^q(\Gamma,V)$ holds due to \eqref{4.15-0}, as we have proved. Thus, the embedding $H^\varphi(\Gamma, V)\hookrightarrow C^q(\Gamma,V)$ is compact.

Let us build the indicated function $\varphi_0$. Put
\begin{equation*}
\eta(t) := \int\limits_t^\infty\frac{\tau^{2q+n-1}}{\varphi^2(\tau)}\,d\tau<\infty
\quad\mbox{and}\quad
\varphi_0(t) := \varphi(t) \sqrt[4]{\eta(t)}
\end{equation*}
for every $t\geq1$. Then $\varphi_0(t)/\varphi(t)=\sqrt[4]{\eta(t)}\rightarrow0$ as $t\rightarrow\infty$. Moreover,
\begin{equation*}
\int\limits_1^\infty\frac{\tau^{2q+n-1}}{\varphi_0^2(\tau)} \,d\tau =
\int\limits_1^\infty\frac{\tau^{2q+n-1}}{\varphi^2(\tau) \sqrt{\eta(\tau)}}\,d\tau=
- \int\limits_1^\infty\frac{d\eta(\tau)}{\sqrt{\eta(\tau)}} =
\int\limits^{\eta(1)}_0\frac{d\eta}{\sqrt{\eta}}<\infty;
\end{equation*}
i.e., $\varphi_0$ satisfies \eqref{4.15-0}. It remains to show that $\varphi_0\in \mathrm{OR}$.
It suffices to prove that $\eta\in \mathrm{OR}$. Let numbers $a>1$ and $c\geq1$ be the same as that in \eqref{f3}.  Given $\lambda\in [1,a]$ and $t\geq1$, we have
\begin{equation*}
\eta(\lambda t)=\int\limits_{\lambda t}^\infty\frac{\tau^{2q+n-1}}{\varphi^2(\tau)}\,d\tau =
\int\limits_{t}^\infty
\frac{(\lambda\theta)^{2q+n-1}}{\varphi^2(\lambda\theta)}\,
\lambda d\theta.
\end{equation*}
Here, owing to \eqref{f3}, we have the inequalities
\begin{equation*}
\frac{c^{-2}}{\varphi^2(\theta)}
\leq\frac{c^{-2}\lambda^{2q+n}}{\varphi^2(\theta)}
\leq\frac{\lambda^{2q+n}}{\varphi^2(\lambda\theta)}\leq \frac{c^2\lambda^{2q+n}}{\varphi^2(\theta)}\leq \frac{c^2a^{2q+n}}{\varphi^2(\theta)}
\end{equation*}
whenever $\theta\geq1$. Therefore,
\begin{equation*}
c^{-2}\eta(t)\leq\eta(\lambda t)\leq c^2a^{2q+n}\eta(t),
\end{equation*}
whenever $t\geq1$ and $q\in [1,a]$. Thus, $\eta\in \mathrm{OR}$, which implies the required inclusion $\varphi_0\in \mathrm{OR}$.

It is remains to prove that the embedding $H^\varphi(\Gamma, V)\hookrightarrow C^q(\Gamma,V)$ implies condition \eqref{4.15}.
Assume now that $H^\varphi(\Gamma, V)\hookrightarrow C^q(\Gamma,V)$ holds.
Without loss of generality we suppose that $\Gamma_1\not\subset \cup_{j=2}^\varkappa \Gamma_j$. Choose an open nonempty set  $U\subset\Gamma_1$ such that $U\cap\Gamma_j=\emptyset$ whenever $j\neq1$, and put $G:= \alpha^{-1}_1(U)$.
Consider an arbitrary distribution $\omega\in H^{\varphi}(\mathbb{R}^n)$ such that $\mathrm{supp}\,\omega\subset G$.
Owing to \eqref{3.14} and our assumption, we have the inclusion
\begin{equation*}
u:=K(\omega, \underbrace{0, \ldots, 0}_{p\varkappa-1}) \in H^{\varphi}(\Gamma, V) \hookrightarrow C^q(\Gamma,V),
\end{equation*}
with $K$ being the sewing operator from the proof of Theorem~\ref{t4.1}.
Hence,
\begin{equation*}
(\omega, \underbrace{0, \ldots, 0}_{p-1}) = (\chi_1 u_1)\circ\alpha_1 \in (C^q(\mathbb{R}^n))^p.
\end{equation*}
(The latter equality is valid because $\chi_1 = 1$ on $U$). Thus, \eqref{6.1} holds, which implies condition~\eqref{4.15} due to Proposition~\ref{p3}.
\end{proof}

\section{Applications}\label{sec7appl}

We will give applications of the extended Sobolev scale to mixed-order elliptic pseudodifferential operators (PsDOs) on a pair of vector bundles over the closed manifold $\Gamma$.

Let $\pi_1: V_1 \rightarrow\Gamma$ and $\pi_2: V_2 \rightarrow\Gamma$ be infinitely smooth complex vector bundles of the same rank $p\geq1$ on $\Gamma$. Following \cite[Definition 18.1.32]{Hermander07v3}, we let $\Psi^m(\Gamma;V_1,V_2)$ denote the class of all PsDOs $A:C^{\infty}(\Gamma,V_1)\to C^{\infty}(\Gamma,V_2)$ of order $m\in \mathbb{R}$. Note that $\Psi^m(\Gamma;V_1,V_2)\subset \Psi^{r}(\Gamma;V_1,V_2)$ whenever $m<r$, and put
\begin{equation*}
\Psi^{\infty}(\Gamma;V_1,V_2):=
\bigcup_{m\in\mathbb{R}}\Psi^m(\Gamma;V_1,V_2)
\quad\mbox{and}\quad
\Psi^{-\infty}(\Gamma;V_1,V_2):=
\bigcap_{m\in\mathbb{R}}\Psi^m(\Gamma;V_1,V_2).
\end{equation*}
The class $\Psi^{-\infty}(\Gamma;V_1,V_2)$ consists of all integral operators with $C^{\infty}$-kernels (cf. \cite[Section~2.1, p.~23]{Agranovich94}).

Suppose that these vector bundles are Hermitian. Let $\langle u,v\rangle_{\Gamma, V_r}$ denote the corresponding inner product of sections $u,v\in C^\infty(\Gamma,V_r)$, with $r\in\{1,2\}$. Given a PsDO $A\in\Psi^{\infty}(\Gamma;V_1,V_2)$, we define its adjoint PsDO $A^{\ast}\in\Psi^{\infty}(\Gamma;V_2,V_1)$ by the formula $\langle Au,w\rangle_{\Gamma,V_2}=\langle u, A^{\ast}w\rangle_{\Gamma, V_1}$ for all $u\in C^\infty(\Gamma, V_1)$ and $w\in C^\infty(\Gamma, V_2)$. The operator $A^{\ast}$ exists and unique; if
$A\in\Psi^{m}(\Gamma;V_1,V_2)$ for certain $m\in\mathbb{R}$, then $A^{\ast}\in\Psi^{m}(\Gamma; V_2, V_1)$ \cite[Chapter~IV, Theorem~3.16(b)]{Wells}. The mapping $u\mapsto Au$, where $u\in C^{\infty}(\Gamma,V_1)$, extends uniquely to a continuous linear operator from $\mathcal{D}'(\Gamma,V_1)$ to $\mathcal{D}'(\Gamma,V_2)$. Thus, the equality $\langle Au,w\rangle_{\Gamma,V_2}=\langle u, A^{\ast}w\rangle_{\Gamma, V_1}$ extends by continuity over all $u\in\mathcal{D}'(\Gamma,V_1)$ and $w\in C^\infty(\Gamma, V_2)$. This equality defines the image $Au\in\mathcal{D}'(\Gamma,V_2)$ of any generalized section $u\in\mathcal{D}'(\Gamma,V_1)$.

\begin{lemma}\label{lem-PsDO-bounded}
Let $m\in\mathbb{R}$ and $A\in\Psi^{m}(\Gamma;V_1,V_2)$. Then $A$ is a bounded operator on the pair of Hilbert spaces
\begin{equation*}
A:H^{\omega}(\Gamma;V_1,V_2)\to H^{\omega\varrho^{-m}}(\Gamma;V_1,V_2)
\end{equation*}
for every $\omega\in\mathrm{OR}$.
\end{lemma}

Here and below, we use the function parameter $\rho(t):=t$ of $t\geq1$ in order not to indicate the variable $t$ in smoothness indexes of distribution spaces. Thus, if $\omega\in\mathrm{OR}$, then  $\omega\varrho^{-m}$ is the function $\omega(t)t^{-m}$ of $t$ and of class $\mathrm{OR}$.

\begin{proof}[Proof of Lemma~$\ref{lem-PsDO-bounded}$]
This lemma is known in the Sobolev case where $\omega(t)\equiv t^s$ for some $s\in\mathbb{R}$ \cite[Chapter~IV, Proposition~3.12]{Wells}. Putting $s_0 := \sigma_0(\omega)-1$ and $s_1 := \sigma_1(\omega)+1$, we have two bounded operators
\begin{equation*}
A: H^{(s_j)}(\Gamma, V_1)\rightarrow H^{(s_j-m)}(\Gamma, V_2), \quad j\in\{0,1\}.
\end{equation*}
Define the function parameter $\psi$ by formula \eqref{3.30} in which $\varphi:=\omega$. We conclude by Theorem~\ref{t4.1} that $A$ is a bounded operator between the spaces
\begin{equation*}
H^{\omega}(\Gamma,V_1)=
\bigl[H^{(s_0)}(\Gamma,V_1),H^{(s_1)}(\Gamma,V_1)\bigr]_\psi
\end{equation*}
and
\begin{equation*}
H^{\omega\rho^{-m}}(\Gamma,V_2)=
\bigl[H^{(s_0-m)}(\Gamma,V_2),H^{(s_1-m)}(\Gamma,V_2)\bigr]_\psi.
\end{equation*}
\end{proof}

We suppose in the sequel that the above vector bundles are decomposed into the direct sums
\begin{equation*}
V_{1}=\bigoplus_{k=1}^{k_\ast}V_{1,k}\quad\mbox{and}\quad
V_{2}=\bigoplus_{j=1}^{j_\ast}V_{2,j}
\end{equation*}
of infinitely smooth Hermitian complex vector bundles $\pi_{1,k}:V_{1,k}\to\Gamma$ and $\pi_{2,j}:V_{2,j}\to\Gamma$, resp. Here, $k_\ast$ and $j_\ast$ are certain integers subject to $1\leq k_\ast\leq p$ and $1\leq j_\ast\leq p$.

Let $A\in\Psi^{\infty}(\Gamma;V_1,V_2)$. Then $A$ is a $j_{\ast}\times k_\ast$-matrix formed by PsDOs $A_{j,k}\in\Psi^{\infty}(\Gamma;V_{1,k},V_{2,j})$, where $j=1,\ldots,j_\ast$ and $k=1,\ldots,k_\ast$. Thus, the equation $Au=f$ takes the form
\begin{equation}\label{equation-D-N}
\sum_{k=1}^{k_\ast}A_{j,k}\,u_{k}=f_{j},\quad j=1,\ldots,j_\ast.
\end{equation}
Here, $u=(u_{1},\ldots,u_{k_\ast})\in\mathcal{D}'(\Gamma,V_1)$ with each $u_{k}\in\mathcal{D}'(\Gamma,V_{1,k})$, and  $f=(f_{1},\ldots,f_{j_\ast})\in\mathcal{D}'(\Gamma,V_2)$ with each $f_{j}\in\mathcal{D}'(\Gamma,V_{2,j})$. The adjoint PsDO $A^{\ast}$ is the $k_{\ast}\times j_{\ast}$-matrix formed by the adjoint PsDOs $A_{j,k}^{\ast}\in\Psi^{\infty}(\Gamma;V_{2,j},V_{1,k})$, where $k=1,\ldots,k_\ast$ and $j=1,\ldots,j_\ast$.

Let real numbers $\ell_{1},\ldots,\ell_{j_\ast}$ and $m_{1},\ldots,m_{k_\ast}$ be chosen so that $A_{j,k}\in\Psi^{\ell_{j}+m_{k}}(\Gamma;V_{1,k},V_{2,j})$ for all $j$ and $k$. We suppose that the PsDO $A=(A_{j,k})$ is Douglis--Nirenberg elliptic on $\Gamma$ for these numbers. Equivalent definitions of this ellipticity is given in \cite[Theorem 19.5.3]{Hermander07v3}. (Such a type of the ellipticity was introduced in \cite{DouglisNirenberg55}.) Then the adjoint PsDO $A^{\ast}$ is also Douglis--Nirenberg elliptic on $\Gamma$ but for numbers $m_{1},\ldots,m_{k_\ast}$ and $\ell_{1},\ldots,\ell_{j_\ast}$. Hence, the spaces
\begin{equation*}
N:=\{u\in C^\infty(\Gamma,V_1):Au=0\;\mbox{on}\;\Gamma\}\quad\mbox{and}\quad
M:=\{w\in C^\infty(\Gamma,V_2):A^{\ast}=0\;\mbox{on}\;\Gamma\}
\end{equation*}
are finite-dimensional \cite[Theorem 19.5.3]{Hermander07v3}.

\begin{theorem}\label{th-Fredholm}
The PsDO $A$ is a Fredholm bounded operator on the pair of Hilbert spaces \begin{equation}\label{DN-operator}
A:\bigoplus_{k=1}^{k_\ast}H^{\varphi\rho^{m_k}}(\Gamma,V_{1,k})\to \bigoplus_{j=1}^{j_\ast}H^{\varphi\rho^{-\ell_j}}(\Gamma,V_{2,j})
\end{equation}
for every $\varphi\in\mathrm{OR}$. The kernel of operator \eqref{DN-operator} equals $N$, and the range
\begin{equation}\label{range-A}
A\biggl(\bigoplus_{k=1}^{k_\ast}H^{\varphi\rho^{m_k}}(\Gamma,V_{1,k})
\biggr)= \biggl\{
f\in\bigoplus_{j=1}^{j_\ast}H^{\varphi\rho^{-\ell_j}}(\Gamma,V_{2,j}):
\langle f,w\rangle_{\Gamma,V_2}=0 \;\,\mbox{for all}\;\,w\in M\biggr\}.
\end{equation}
Hence, the index of this operator is equal to $\dim N-\dim M$.
\end{theorem}

Recall that a bounded linear operator $L:E_1\rightarrow E_2$ between Banach spaces $E_1$ and $E_2$ is called Fredholm if its kernel $\ker L:=\{v\in E_1:Lv=0\}$ and co-kernel $\mathrm{coker}\,L:=E_2/L(E_1)$ are finite-dimensional. If the operator $L$ is Fredholm, then its range $L(E_1)$ is closed in $E_2$ \cite[Lemma~19.1.1]{Hermander07v3} and consists of all vectors $g\in E_2$ such that $\Phi(g)=0$ for every functional $\Phi\in\ker L^{\ast}$, and its index satisfies
\begin{equation*}
\mathrm{ind}\,L:=\dim\ker L-\dim\mathrm{coker}\,L=
\dim\ker L-\dim\ker L^{\ast}<\infty.
\end{equation*}
here $L^{\ast}$ is the adjoint operator to $L$.

\begin{proof}[Proof of Theorem~$\ref{th-Fredholm}$]
Owing to Lemma~\ref{lem-PsDO-bounded} applied to $A_{j,k}$, the PsDO $A$ is a bounded operator on the pair of spaces \eqref{DN-operator}. According to \cite[Theorem 19.5.3]{Hermander07v3}, the Douglis--Nirenberg ellipticity of this operator is equivalent to the following property: there exists a PsDO $B\in\Psi^{\infty}(\Gamma;V_2,V_1)$ such that $B$ is $k_{\ast}\times j_{\ast}$-matrix formed by PsDOs $B_{k,j}\in\Psi^{-m_k-\ell_j}(\Gamma;V_{2,j},V_{1,k})$, where $k=1,\ldots,k_\ast$ and $j=1,\ldots,j_\ast$, and that
\begin{equation}\label{parametrix-equalities}
T_1:=BA-I_1\in\Psi^{-\infty}(\Gamma;V_1,V_1)\quad\mbox{and}\quad
T_2:=AB-I_2\in\Psi^{-\infty}(\Gamma;V_2,V_2),
\end{equation}
where $I_r$ is the identity operator on $\mathcal{D}'(\Gamma;V_r,V_r)$ for each $r\in\{1,2\}$. The PsDO $B$ is Douglis--Nirenberg elliptic on $\Gamma$ for numbers $-m_{1},\ldots,-m_{k_\ast}$ and $-\ell_{1},\ldots,-\ell_{j_\ast}$ and is called a two-sided parametrix for $A$. According to Lemma~\ref{lem-PsDO-bounded}, we have bounded operators
\begin{equation}\label{parametrix-operator}
B:\bigoplus_{j=1}^{j_\ast}H^{\varphi\rho^{-\ell_j}}(\Gamma,V_{2,j})\to
\bigoplus_{k=1}^{k_\ast}H^{\varphi\rho^{m_k}}(\Gamma,V_{1,k})
\end{equation}
and
\begin{gather}\label{T-1}
T_1:\bigoplus_{k=1}^{k_\ast}H^{\varphi\rho^{m_k}}(\Gamma,V_{1,k})\to
\bigoplus_{k=1}^{k_\ast}H^{\varphi\rho^{m_k+\lambda}}(\Gamma,V_{1,k}),\\
T_2:\bigoplus_{j=1}^{j_\ast}H^{\varphi\rho^{-\ell_j}}(\Gamma,V_{2,j})\to
\bigoplus_{j=1}^{j_\ast}H^{\varphi\rho^{-\ell_j+\lambda}}(\Gamma,V_{2,j})
\label{T-2}
\end{gather}
for every $\lambda>0$. This implies by Theorem~\ref{t4.3} that $T_1$ and $T_2$ are compact operators on the source spaces of \eqref{T-1} and \eqref{T-2}, resp. Hence \cite[Corollary 19.1.9]{Hermander07v3}, it follows from equalities \eqref{parametrix-equalities} that the bounded operator \eqref{DN-operator} is Fredholm. Its kernel lies in $C^\infty(\Gamma,V_1)$ due to the first equality in \eqref{parametrix-equalities} and property \eqref{T-1}; thus, the kernel is $N$. The adjoint operator to \eqref{DN-operator} coincides with the PsDO $A^{\ast}$ in the sense of Theorem~\ref{t4.6}. Since $A^{\ast}$ is also Douglis--Nirenberg elliptic, $M$ is the kernel of the adjoint of \eqref{DN-operator}, which implies \eqref{range-A}.
\end{proof}

Let $\Gamma_0$ be an open nonempty subset of the manifold $\Gamma$. Considering the elliptic equation $Au=f$ (i.e., the mixed-order system \eqref{equation-D-N}), we study the local regularity of its solution on~$\Gamma_0$.

Given $\omega\in \mathrm{OR}$, $r\in\{1,2\}$, and admissible $i$, we let $H_{\mathrm{loc}}^{\omega}(\Gamma_0,V_{r,i})$ denote the linear space of all $v\in\mathcal{D}'(\Gamma,V_{r,i})$ such that $\chi v\in H^{\omega}(\Gamma,V_{r,i})$ for every scalar function $\chi\in C^\infty(\Gamma)$ with $\mathrm{supp}\,\chi\subset\Gamma_0$. Of course,
if $\Gamma_0=\Gamma$, then $H_{\mathrm{loc}}^{\omega}(\Gamma_0,V_{r,i})$ coincides with $H^{\omega}(\Gamma,V_{r,i})$.

\begin{theorem}\label{th-regularity}
Let $u \in \mathcal{D}'(\Gamma, V_1)$, $f \in \mathcal{D}'(\Gamma, V_2)$, and $\varphi \in \mathrm{OR}$. Assume that $Au=f$ on $\Gamma_0$. Then
\begin{equation}\label{equivalence-regularity}
f\in\bigoplus_{j=1}^{j_\ast}
H^{\varphi\rho^{-\ell_j}}_{\mathrm{loc}}(\Gamma_0,V_{2,j})
\;\; \Longleftrightarrow \;\;
u\in \bigoplus_{k=1}^{k_\ast}
H^{\varphi\rho^{m_k}}_{\mathrm{loc}}(\Gamma_0,V_{1,k}).
\end{equation}
\end{theorem}

\begin{proof}
We choose any function $\chi\in C^\infty(\Gamma)$ and a certain function $\eta\in C^\infty(\Gamma)$ such that
\begin{equation}\label{cut-functions}
\mathrm{supp}\,\chi\subset\Gamma_0,\quad \mathrm{supp}\,\eta\subset\Gamma_0,
\quad\mbox{and}\quad
\eta=1\;\,\mbox{in a neighbourhood of}\;\,\mathrm{supp}\,\chi.
\end{equation}
According to \eqref{parametrix-equalities}, we get
\begin{equation}\label{chi-u}
\chi u=\chi BAu-\chi T_{1}u=\chi B\eta Au+\chi B(1-\eta)Au-\chi T_{1}u
=:\chi B\eta Au+Tu
\end{equation}
where $T\in\Psi^{-\infty}(\Gamma;V_1,V_1)$ because the supports of $\chi$ and $(1-\eta)$ are disjoint and since $T_{1}\in\Psi^{-\infty}(\Gamma;V_1,V_1)$. Hence, if the left-hand side of \eqref{equivalence-regularity} holds true, then
\begin{equation*}
\chi u=\chi B\eta f+Tu\in
\bigoplus_{k=1}^{k_\ast}H^{\varphi\rho^{m_k}}(\Gamma,V_{1,k})
\end{equation*}
in view of \eqref{parametrix-operator}, i.e. $u$ satisfies the right-hand side of \eqref{equivalence-regularity}.

On the other hand, if the right-hand side of \eqref{equivalence-regularity} holds true, then
\begin{equation*}
\chi f=\chi Au=\chi A(1-\eta)u+\chi A\eta u\in
\bigoplus_{j=1}^{j_\ast}H^{\varphi\rho^{-\ell_j}}(\Gamma,V_{2,j})
\end{equation*}
because the PsDO $u\mapsto\chi A(1-\eta)u$ belong to $\Psi^{-\infty}(\Gamma;V_1,V_2)$ and due to \eqref{DN-operator}, i.e. $f$ satisfies the left-hand side of \eqref{equivalence-regularity}.
\end{proof}

We supplement Theorem~\ref{th-regularity} with a local \textit{a~priory} estimate of $u$.

\begin{theorem}\label{th-estimate}
Let $\varphi\in\mathrm{OR}$, $\lambda\in\mathbb{R}$, $u\in H^{(\lambda)}(\Gamma,V_1)$, $f\in\mathcal{D}'(\Gamma,V_2)$, and $Au=f$ on~$\Gamma_0$. Suppose that the right-hand (or left-hand) side of \eqref{equivalence-regularity} holds true. Let scalar functions $\chi,\eta\in C^\infty(\Gamma)$ satisfy \eqref{cut-functions}. Then
\begin{equation}\label{local-estimate}
\|\chi u\|'_{\varphi}\leq
c\,\bigl(\|\eta f\|''_{\varphi}+\|u\|_{\lambda}\bigr),
\end{equation}
where $c$ is a certain positive number that does not depend on $u$ and $f$. Here, $\|\cdot\|'_{\varphi}$ and $\|\cdot\|''_{\varphi}$ denote the norms in the source space and target space of operator \eqref{DN-operator}, whereas $\|\cdot\|_{\lambda}$ stand for the norm in $H^{(\lambda)}(\Gamma,V_1)$. If the operator $A$ is differential, then \eqref{local-estimate} holds true with $\|\eta u\|_{\lambda}$ instead of $\|u\|_{\lambda}$ (and we may replace the hypothesis $u\in H^{(\lambda)}(\Gamma,V_1)$ with $u\in\mathcal{D}'(\Gamma,V_1)$).
\end{theorem}

Of course, estimate \eqref{local-estimate} makes sense if $-\lambda\gg1$; then the source space of operator \eqref{DN-operator} is embedded in $H^{(\lambda)}(\Gamma,V_1)$.

\begin{proof}
The required estimate \eqref{local-estimate} follows from \eqref{chi-u}, \eqref{parametrix-operator} and $T\in\Psi^{-\infty}(\Gamma;V_1,V_1)$:
\begin{equation*}
\|\chi u\|'_{\varphi}\leq \|\chi B\eta Au\|'_{\varphi}+\|Tu\|'_{\varphi}
\leq c\,\bigl(\|\eta Au\|''_{\varphi}+\|u\|_{\lambda}\bigr).
\end{equation*}
Suppose now that the operator $A$ is a differential. Choose a function $\eta_{0}\in C^\infty(\Gamma)$ that satisfies the following conditions: $\mathrm{supp}\,\eta_{0}\subset\Gamma_0$, $\eta_{0}=1$ in a neighbourhood of $\mathrm{supp}\,\chi$, and $\eta=1$ in a neighbourhood of $\mathrm{supp}\,\eta_{0}$. It follows from \eqref{chi-u} that $\chi v=\chi B\eta_{0}Av+T_{0}v$ for every $v\in\mathcal{D}'(\Gamma,V_1)$ and certain $T_{0}\in\Psi^{-\infty}(\Gamma;V_1,V_1)$. Then
\begin{equation*}
\chi u=\chi\eta u=\chi B\eta_{0}A\eta u+T_{0}\eta u=
\chi B\eta_{0}Au+T_{0}\eta u;
\end{equation*}
the last equality is true because $A$ is differential. Hence,
\begin{equation*}
\|\chi u\|'_{\varphi}\leq
\|\chi B\eta_{0}Au\|'_{\varphi}+\|T_{0}\eta u\|'_{\varphi}
\leq c_{0}\bigl(\|\eta_{0}Au\|''_{\varphi}+\|\eta u\|_{\lambda}\bigr),
\end{equation*}
with
\begin{equation*}
\|\eta_{0}Au\|''_{\varphi}=\|\eta_{0}\eta Au\|''_{\varphi}\leq
c_{1}\|\eta Au\|''_{\varphi}.
\end{equation*}
Here, the positive numbers $c_{0}$ and $c_{1}$ do not depend on $u$ (and $f$). This yields the required estimate.
\end{proof}

\begin{remark}\label{rem-estimate}
If we replace the hypothesis $u\in H^{(\lambda)}(\Gamma,V_1)$ with
\begin{equation}\label{-1-regularity}
u\in\bigoplus_{k=1}^{k_\ast}H^{\varphi\rho^{m_k-1}}(\Gamma,V_{1,k})
\end{equation}
in Theorem~\ref{th-estimate} and retains other hypotheses, then
\begin{equation}\label{-1-local-estimate}
\|\chi u\|'_{\varphi}\leq
c\,\bigl(\|\chi f\|''_{\varphi}+\|u\|'_{\varphi\rho^{-1}}\bigr).
\end{equation}
Of course, $\|\cdot\|'_{\varphi\rho^{-1}}$ stands for the norm in the space written in \eqref{-1-regularity}. If, in addition, $A$ is a differential operator, then \eqref{-1-local-estimate} holds true with $\|\eta u\|'_{\varphi\rho^{-1}}$ instead of $\|u\|'_{\varphi\rho^{-1}}$ (and we may replace \eqref{-1-regularity} with $u\in\mathcal{D}'(\Gamma,V_1)$).
\end{remark}

\begin{proof}
According to Theorem~\ref{th-estimate} (considered for $\Gamma_{0}=\Gamma$ and $\chi\equiv\eta\equiv1$) there exists a number $c_1>0$ such that
\begin{equation*}
\|v\|'_{\varphi}\leq
c_1\bigl(\|Av\|''_{\varphi}+\|v\|'_{\varphi\rho^{-1}}\bigr)
\end{equation*}
for every section $v$ from the space indicated in \eqref{-1-regularity}. Assuming \eqref{-1-regularity} and putting $v:=\chi u$, we hence obtain the estimate
\begin{equation}\label{rem-estimate-1}
\|\chi u\|'_{\varphi}\leq
c_1\bigl(\|A\chi u\|''_{\varphi}+\|\chi u\|'_{\varphi\rho^{-1}}\bigr).
\end{equation}
Rearranging the PsDO $A$ and the operator of the multiplication by $\chi$, we write $A\chi u=\chi Au+A'u$, where $A'$ is a certain PsDO of class $\Psi^{m-1}(\Gamma;V_1,V_2)$ (see, e.g., \cite[p.~13]{Agranovich94}). Therefore,
\begin{equation}\label{rem-estimate-2}
\|A\chi u\|''_{\varphi}\leq
\|\chi Au\|''_{\varphi}+\|A'u\|''_{\varphi}\leq
\|\chi Au\|''_{\varphi}+c_2\|u\|'_{\varphi\rho^{-1}},
\end{equation}
where $c_2$ is the norm of $A'$ on the corresponding spaces. Now \eqref{rem-estimate-1} and \eqref{rem-estimate-2} yield the required estimate \eqref{-1-local-estimate}.

If the operator $A$ is differential, then we write
\begin{equation*}
A\chi u=A\chi\eta u=\chi A\eta u+A'\eta u=\chi Au+A'\eta u
\end{equation*}
and hence get the estimate
\begin{equation*}\label{rem-estimate-3}
\|A\chi u\|''_{\varphi}\leq
\|\chi Au\|''_{\varphi}+\|A'\eta u\|''_{\varphi}\leq
\|\chi Au\|''_{\varphi}+c_2\|\eta u\|'_{\varphi\rho^{-1}}
\end{equation*}
and note that
\begin{equation*}
\|\chi u\|'_{\varphi\rho^{-1}}=\|\chi \eta u\|'_{\varphi\rho^{-1}}\leq c_{3}\|\eta u\|'_{\varphi\rho^{-1}},
\end{equation*}
where the number $c_{3}>0$ does not depend on $u$. This together with \eqref{rem-estimate-1} substantiates the last sentence of Remark~\ref{rem-estimate}.
\end{proof}

Ending this section, we give an exact sufficient condition for a component $u_{k}\in\mathcal{D}'(\Gamma,V_{1,k})$ of the solution $u=(u_{1},\ldots,u_{k_\ast})\in\mathcal{D}'(\Gamma,V_1)$ of equation \eqref{equation-D-N} to be $q$ times continuously differentiable on $\Gamma_0$. This property is symbolized as $u\in C^{q}(\Gamma_0,V_{1,k})$.

\begin{theorem}\label{th-classical-smoothness}
Let $k\in\{1,\ldots,k_\ast\}$ and $0\leq q\in\mathbb{Z}$, and let $\varphi\in\mathrm{OR}$ satisfy
\begin{equation}\label{condition-classical-smoothness}
\int\limits_1^\infty t^{2q+n-1-2k}\,\varphi^{-2}(t)\,dt<\infty.
\end{equation}
Suppose that $u \in \mathcal{D}'(\Gamma, V)$ is a solution to the elliptic equation $Au=f$ on $\Gamma_0$ for certain
\begin{equation*}
f\in\bigoplus_{j=1}^{j_\ast}
H^{\varphi\rho^{-\ell_j}}_{\mathrm{loc}}(\Gamma_0,V_{2,j}).
\end{equation*}
Then $u_{k}\in C^{q}(\Gamma_0,V_{1,k})$.
\end{theorem}

\begin{proof}
We arbitrarily choose a point $x\in\Gamma_0$ and a certain function $\chi\in C^{\infty}(\Gamma)$ satisfying $\mathrm{supp}\,\chi\subset \Gamma_0$ and $\chi=1$ in a neighbourhood $\Gamma_{x}$ of $x$. We conclude by Theorems \ref{th-regularity} and \ref{t4.7}  that
\begin{equation*}
\chi u_{k}\in H^{\varphi\rho^{m_k}}(\Gamma,V_{1,k})\subset
C^{q}(\Gamma,V_{1,k}).
\end{equation*}
Thus $u_{k}\in C^{q}(\Gamma_{x},V_{1,k})$ for every $x\in\Gamma_0$, which gives the required inclusion.
\end{proof}

\begin{remark}
Condition \eqref{condition-classical-smoothness} is exact in this theorem because this condition follows from the implication
\begin{equation}\label{implication-C^q}
\biggl(u \in \mathcal{D}'(\Gamma, V_1),\; Au\in \bigoplus_{j=1}^{j_\ast} H^{\varphi\rho^{-\ell_j}}_{\mathrm{loc}}(\Gamma_0,V_{2,j})\biggr)
\Longrightarrow u_{k}\in C^{q}(\Gamma_0,V_{1,k}).
\end{equation}
\end{remark}

\begin{proof}
Suppose that implication \eqref{implication-C^q} is true. If
\begin{equation*}
u\in\bigoplus_{k=1}^{k_\ast}
H^{\varphi\rho^{m_k}}_{\mathrm{loc}}(\Gamma_0,V_{1,k}),
\end{equation*}
then $u$ satisfies the premise of this implication by Theorem~\ref{th-regularity}. Hence,
\begin{equation}\label{H-C-remark}
H^{\varphi\rho^{m_k}}_{\mathrm{loc}}(\Gamma_0,V_{1,k})\subset
C^{q}(\Gamma_0,V_{1,k})
\end{equation}
due to \eqref{implication-C^q}. It remains to show that \eqref{H-C-remark} implies condition \eqref{condition-classical-smoothness} by Proposition~\ref{p3}. We assume without loss of generality that  $\Gamma_0\cap\Gamma_1\neq\emptyset$ and
choose a nonempty open set $G\subset\mathbb{R}^n$ such that $\overline{\alpha_{1}(G)}\subset\Gamma_0\cap\Gamma_1$. Recall that $\alpha_{1}:\mathbb{R}^{n}\to\Gamma_{1}$ is the above-mentioned local chart of $\Gamma$. Consider an arbitrary function $w_1\in H^{\varphi\rho^{m_k}}(\mathbb{R}^n)$ such that $\mathrm{supp}\,w_1\subset G$, and put $w:=(w_1,0)\in (H^{\varphi\rho^{m_k}}(\mathbb{R}^n))^{p_k}$, where $p_k:=\dim V_{1,k}$. There exists a section $v\in H^{\varphi\rho^{m_k}}(\Gamma, V_{1,k})$ such that $\mathrm{supp}\,v\subset\alpha_{1}(G)$ and $v^{\times}_{1}\circ\alpha_{1}=w$, where $(v^{\times}_{1},\ldots,v^{\times}_{\varkappa})$ is the representation of $v$ in the local trivializations of the vector bundle
$\pi_{1,k}:V_{1,k}\to\Gamma$. Then $v\in C^{q}(\Gamma_0,V_{1,k})$ by \eqref{H-C-remark}, which implies that $w_1\in C^{q}(\mathbb{R}^n)$. Thus,
\begin{equation*}
\{w_1\in H^{\varphi\rho^{m_k}}(\mathbb{R}^n):\mathrm{supp}\,w_1\subset G\} \subset C^q(\mathbb{R}^n).
\end{equation*}
This entails \eqref{condition-classical-smoothness} by Proposition~\ref{p3}.
\end{proof}

\subsection*{Acknowledgments.} The first named author was funded by the National Academy of Sciences of Ukraine, a grant from the Simons Foundation (1030291, A.A.M., and 1290607, A.A.M.), and Universities-for-Ukraine (U4U) Non-residential Fellowship granted by UC Berkeley Economics/Haas. The second named author was funded by the program "Short-Term Grants" of DAAD. The work was also supported by the European Union's Horizon 2020 research and innovation programme under the Marie Sk{\l}odowska-Curie grant agreement No~873071 (\nobreak{SOMPATY}: Spectral Optimization: From Mathematics to Physics and Advanced Technology).

\end{document}